\documentclass[10pt]{amsart}

\usepackage[english]{babel}
\usepackage{amsthm}
\usepackage{amssymb}
\usepackage{graphicx}
\usepackage{subfigure}
\usepackage{geometry} 
\usepackage[all]{xy}
\usepackage{mathdots} 

\usepackage{tikz}
\usetikzlibrary{matrix,arrows,calc}
\usetikzlibrary{positioning}
 \usepackage[section]{placeins}

\usepackage{csquotes}
\usepackage{mathabx}
\usepackage{paralist}
\usepackage{makecell}

\newtheorem{prop}{Proposition}[section]
\newtheorem{thm}[prop]{Theorem}
\newtheorem{cor}[prop]{Corollary}
\newtheorem{lem}[prop]{Lemma}
\newtheorem{ex}[prop]{Example}
  
\theoremstyle{definition}
\newtheorem{rem}[prop]{Remark}

\newtheorem{defin}[prop]{Definition}

\newcommand{\veq}{\mathrel{\rotatebox{90}{$=$}}}

\newcommand{\abs}[1]{\ensuremath{\left\vert#1\right\vert}}
\newcommand\restr[2]{{
  \left.\kern-\nulldelimiterspace 
  #1 
  \vphantom{|} 
  \right|_{#2} 
  }}

\usepackage{chngcntr}
\counterwithin{figure}{section}

\usepackage{float}

\definecolor{myblue}{cmyk}{1.00,0.56,0.00,0.34}
\definecolor{mygreen}{cmyk}{0.5,0,0.5,0.5}
\definecolor{myred}{cmyk}{0.00,1.00,0.63,0.00}
\definecolor{myyellow}{cmyk}{0.00,0.15,1.00,0.00}

\begin{document}

\title{Arithmetic infinite friezes from punctured discs}
\author{Manuela Tschabold}

\begin{abstract}

We define the notion of infinite friezes of positive integers as a variation of Conway-Coxeter frieze patterns and study their properties. We introduce useful gluing and cutting operations on infinite friezes. It turns out that triangulations of once-punctured discs give rise to periodic infinite friezes having special properties, a notable example being that each diagonal consists of a collection of arithmetic progressions. Furthermore, we work out a combinatorial interpretation of the entries of infinite friezes associated to triangulations of once-punctured discs via matching numbers for certain combinatorial objects, namely periodic triangulations of strips. Alternatively, we consider a known algorithm that as we show computes as well these entries.
\end{abstract}

\maketitle

%
\section*{Introduction} \let\thefootnote\relax\footnotetext{\emph{Key words and phrases:} Infinite frieze, frieze pattern, arithmetic progression,  triangulation, matching number, cluster algebra.}
%

Frieze patterns in mathematics were introduced and studied in \cite{{CCI}} by Conway and Coxeter. These are patterns of $n$ bi-infinite rows of positive integers, bounded from above and below by a row of $0$'s followed by a row of $1$'s, whose entries satisfy a local rule. More precisely, the rows are shifted such that the entries form a diamond pattern, where every such diamond
\begin{center}
\begin{tikzpicture}[font=\normalsize] 
  \matrix(m) [matrix of math nodes,row sep={1em,between origins},column sep={1em,between origins},nodes in empty cells]{
&b&\\
a&&d\\
&c&\\
     };
\end{tikzpicture}
\end{center}
satisfies $ad-bc=1$. Such patterns are periodic in the horizontal direction and also invariant under a glide reflection. In this article we refer to them as \emph{finite friezes}. There are several connections between finite friezes and classical objects in mathematics. A well known correspondence between finite friezes and triangulated polygons, first conjectured in \cite{{C}} and proved in \cite{{CCI}}, is that every finite frieze arises from a triangulated polygon and vice versa, providing a geometric interpretation of the first non-trivial row of a finite frieze via matching numbers between vertices and triangles. Later this was extended for all entries in a finite frieze by Broline, Crowe and Issacs in \cite{{BCI}}.

By the work of Caldero and Chapoton in \cite{{CC}}, finite friezes are closely related to Fomin-Zelevinsky cluster algebras of type $A$. This fact serves as motivation for the work in this article. Various generalizations of finite friezes have been introduced and studied recently, providing new information about cluster algebras, for example frieze patterns of type $D$ in \cite{{BM}}, frieze patterns from higher angulations \cite{{BHJ}}, friezes in \cite{{ARS}}, $\mathrm{SL}_2$-tilings in \cite{{ARS},{BR},{HJ}}, $2$-frieze patterns in \cite{{MOT},{P}}. 

In this article, we generalize and extend the notion of finite friezes and introduce similar patterns of positive integers without the condition of bounding rows at the bottom. We call them \emph{infinite friezes}.  Various properties known for finite friezes can be adapted to infinite friezes. We shall present some of them. Unlike finite friezes, infinite friezes are not necessarily periodic. Our main result is that triangulations of once-punctured discs give rise to periodic infinite friezes, providing a connection between triangulations and infinite friezes. Moreover, for these particular periodic infinite friezes arising from triangulations of once-punctured discs, we are able to give a combinatorial interpretation of the numbers occurring in them via matching numbers. In \cite{BPT} we complete this work and obtain a characterization of infinite friezes via triangulations.

This article is organized as follows. In Section~\ref{secfriezes}, we introduce infinite friezes and give some useful relations between the entries in them (Lemma~\ref{lemrelationqsdiag}). We will focus on a special class of infinite friezes which are invariant under horizontal translation, called \emph{periodic infinite friezes}. In Section~\ref{seccutglue}, we define two algebraic operations on infinite friezes, namely \emph{gluing} and \emph{cutting} (Theorems~\ref{thmgluefrieze}, \ref{thmcutfrieze}). Moreover, we extend these operations to the periodic case by introducing $n$-gluing and $n$-cutting (Propositions~\ref{propnglueperiodicfrieze}, \ref{propncutperiodicfrieze}). The latter of which provides a powerful tool for inductive proofs. We begin Section~\ref{secarithmeticfriezes}, by recalling the definition of triangulations of once-punctured discs, after which we explain how particular periodic infinite friezes arise from such triangulations. More precisely, a triangulations of a once-punctured disc yields a sequence of non-negative integers in a similar manner as for triangulations of polygons \cite{{CCI}}. We call such sequences \emph{quiddity sequences}. We prove that these sequences arising from triangulations of once-punctured discs give rise to periodic infinite friezes (Theorem~\ref{thmtriangulationfrieze}). Furthermore, we establish a remarkable property of such periodic infinite friezes, namely that they exhibit (multiple) arithmetic progressions within each diagonal (where the number depends on the period), and are thus examples of so-called \emph{arithmetic friezes} (Proposition~\ref{proptriangulationfriezearithmetic}). In Section~\ref{secmachingnumbers}, we explain how triangulations of once-punctured discs correspond to periodic triangulations of a certain combinatorial structure we call the \emph{strip} (Theorem~\ref{thmbijection}). A similar model was introduced by Holm and J{\o}rgensen in \cite{{HJ}} in order to describe a certain family of $\mathrm{SL}_2$-\emph{tilings}, which are bi-infinite arrays of positive integers satisfying the same local rule as finite and infinite friezes. There is a subtle but significant difference between the two approaches, which we expand upon in the final section. In our approach, we generalize one of the main results of Broline, Crowe and Issacs from \cite{{BCI}}, giving a way of obtaining an arbitrary entry of an arithmetic frieze from the associated periodic triangulation of the strip via matching numbers (Theorem~\ref{thmmatchings}). The terminology of matching numbers was also used by Baur and Marsh (in \cite{{BM}}) to construct frieze pattern of type $D$. Finally, in Section~\ref{seclabelingalgorithm}, we give an alternative description of an arbitrary entry in an arithmetic frieze using a similar method of assigning labels to the vertices of a periodic triangulation of a strip to that used by Conway and Coxeter (in \cite{{CCI}}) for triangulation of polygons (Theorem~\ref{thmaltderscription}). This labeling algorithm also provides the common differences for the arithmetic progressions of an arithmetic frieze (Proposition~\ref{propcommondifferences}).

%
\section{Periodic infinite friezes}\label{secfriezes}
%

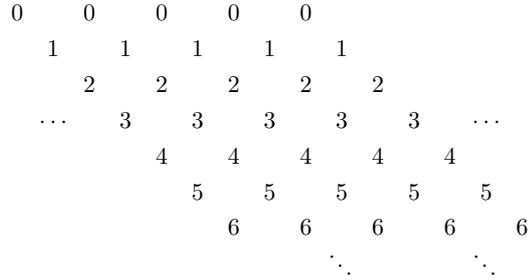
\begin{figure}[t]
\scalebox{.9}{\begin{tikzpicture}[font=\normalsize] 
  \matrix(m) [matrix of math nodes,row sep={1.5em,between origins},column sep={1.5em,between origins},nodes in empty cells]{
0&&0&&0&&0&&0&&&&&&\\
&1&&1&&1&&1&&1&&&&&\\
&&2&&2&&2&&2&&2&&&&\\
&\node{\cdots};&&3&&3&&3&&3&&3&&\node{\cdots};&\\
&&&&4&&4&&4&&4&&4&&\\
&&&&&5&&5&&5&&5&&5&\\
&&&&&&6&&6&&6&&6&&6\\
&&&&&&&&&\node[rotate=-6.5,shift={(-0.034cm,-0.08cm)}]  {\ddots};&&&&\node[rotate=-6.5,shift={(-0.034cm,-0.08cm)}]  {\ddots};&\\
};
\end{tikzpicture}}
\caption{The basic infinite frieze $\mathcal{F}_\ast$.}\label{figbasicfrieze}
\end{figure}

\begin{defin}\label{definffrieze}
An \emph{infinite frieze} $\mathcal{F}$ is an array $(m_{ij})_{i,j\in\mathbb{Z},j-i\ge -2}$ of shifted infinite rows of positive integers bounded at the top by a row filled with $0$'s, followed by a row of $1$'s, i.e.\ $m_{i,i-2}=0$, $m_{i,i-1}=1$ for all $i\in\mathbb{Z}$, and $m_{ij}>0$ otherwise,
\begin{center}
\begin{tikzpicture}[font=\normalsize] 
  \matrix(m) [matrix of math nodes,row sep={1.75em,between origins},column sep={1.75em,between origins},nodes in empty cells]{
&0&&0&&0&&0&&0&&&&&\\[-0.25em]
\node{\cdots};&&1&&1&&1&&1&&1&&\node{\cdots};&&\\[-0.25em]
&&&m_{-2,-2}&&m_{-1,-1}&&m_{00}&&m_{11}&&m_{22}&&&\\
&&\node{\cdots};&&m_{-2,-1}&&m_{-1,0}&&m_{01}&&m_{12}&&m_{23}&&\node{\cdots};\\
&&&&&m_{-2,0}&&m_{-1,1}&&m_{02}&&m_{13}&&m_{24}&\\
&&&&&&&&\node[rotate=-6.5,shift={(-0.034cm,-0.08cm)}]  {\ddots};&&&&\node[rotate=-6.5,shift={(-0.034cm,-0.08cm)}]  {\ddots};&&\\
     };
\end{tikzpicture}
\end{center}
\noindent such that the \emph{unimodular rule} is satisfied, i.e.\ for every diamond in $\mathcal{F}$ of the form
\begin{center}
\begin{tikzpicture}[font=\normalsize] 
  \matrix(m) [matrix of math nodes,row sep={1.75em,between origins},column sep={1.75em,between origins},nodes in empty cells]{
&m_{i+1,j}&\\
m_{ij}&&m_{i+1,j+1}\\
&m_{i,j+1}&\\
     };
\end{tikzpicture}
\end{center}
\noindent the relation $m_{ij}m_{i+1,j+1}-m_{i+1,j}m_{i,j+1}=1$ holds, where $j-i\geq -1$.
\end{defin}

The easiest example of an infinite frieze we may think of is the \emph{basic infinite frieze}\linebreak \mbox{$\mathcal{F}_\ast=(m_{ij})_{j-i\ge -2}$} with constant rows given by $m_{ij}=j-i+2$ as shown in Figure~\ref{figbasicfrieze}.

One can easily convince oneself that if an entry $1$ appears in a non-trivial row of an infinite frieze, it is not possible that it has a $1$ as a neighbor entry to the left, or right, respectively. Moreover, the product of the two neighboring entries of the entry $1$ is strictly bigger than $4$.

\begin{defin}\label{defquidseq}
For an infinite frieze $\mathcal{F}=(m_{ij})_{j-i\ge -2}$ its \emph{quiddity row} is the infinite sequence $(a_i)_{i\in\mathbb{Z}}$ of positive integers given by the first non-trivial row of $\mathcal{F}$, where $a_i=m_{ii}$.
\end{defin}

\begin{rem}
Clearly, two consecutive rows of an infinite frieze, except the first two rows, determine the rest of the infinite frieze, since we can fill the next row below and above these two by using the unimodular rule. It follows that an infinite frieze is determined by its quiddity row.
\end{rem}

\begin{figure}[b]
\resizebox{.9\linewidth}{!}{\begin{tikzpicture}[font=\normalsize] 
  \matrix(m) [matrix of math nodes,row sep={1.5em,between origins},column sep={1.5em,between origins},nodes in empty cells]{
&&&&&&&&&&&&&&&&&&&&&&&&&&&&&\\
0&&0&&0&&0&&0&&0&&0&&0&&&&&&&&&&&&&&&\\
&1&&1&&1&&1&&1&&1&&1&&1&&&&&&&&&&&&&&\\
&&1&&5&&4&&1&&3&&1&&5&&4&&&&&&&&&&&&&&\\
&&&4&&19&&3&&2&&2&&4&&19&&3&&&&&&&&&&&&&\\
&&&&15&&14&&5&&1&&7&&15&&14&&5&&&&&&&&&&&\\
&&\node{\cdots};&&&11&&23&&2&&3&&26&&11&&23&&2&&&\node{\cdots};&&&&&&&\\
&&&&&&18&&9&&5&&11&&19&&18&&9&&5&&&&&&&&&\\
&&&&&&&7&&22&&18&&8&&31&&7&&22&&18&&&&&&&&\\
&&&&&&&&17&&79&&13&&13&&12&&17&&79&&13&&&&&&&\\
&&&&&&&&&61&&57&&21&&5&&29&&61&&57&&21&&&&&&\\
&&&&&&&&&&44&&92&&8&&12&&104&&44&&92&&8&&&&&\\
&&&&&&&&&&&71&&35&&19&&43&&75&&71&&35&&19&&&&\\
&&&&&&&&&\node{\cdots};&&&27&&83&&68&&31&&121&&27&&83&&68&&&\node{\cdots};\\
&&&&&&&&&&&&&64&&297&&49&&50&&46&&64&&297&&49&&\\
&&&&&&&&&&&&&&229&&214&&79&&19&&109&&229&&214&&79&\\
&&&&&&&&&&&&&&&165&&345&&30&&45&&390&&165&&345&&30\\
&&&&&&&&&&&&&&&&&&&&&&&&&&&&&\\
&&&&&&&&&&&&&&&&&&&&&\node[rotate=-6.5,shift={(-0.034cm,-0.08cm)}]  {\ddots};&&&&&&\node[rotate=-6.5,shift={(-0.034cm,-0.08cm)}]  {\ddots};&&\\
};
     
\draw (m-18-23) node[myblue,shift={(0cm,0.1cm)}]{$\mathcal{D}$};
\draw[opacity=0,rounded corners,fill=myblue,fill opacity=0.15] (m-1-2.south west) -- (m-17-18.south west) -- (m-17-28.south west) -- (m-1-12.south west) -- cycle;

\end{tikzpicture}}
\caption{An infinite frieze of period $5$ with a fundamental domain $\mathcal{D}$.}\label{figexfrieze5p1}
\end{figure}
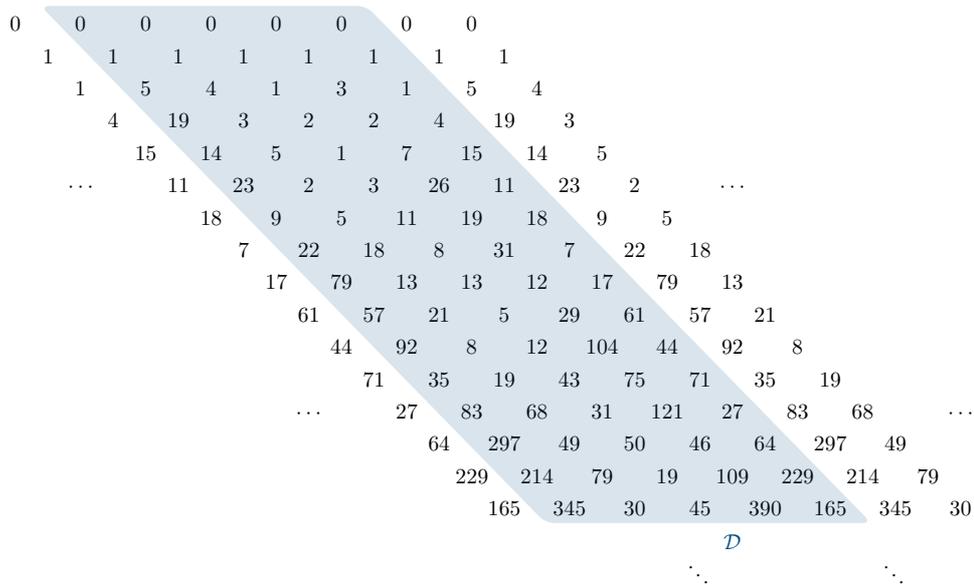 

For our purpose we are mainly interested in a particular family of infinite friezes which stay invariant under horizontal translation, such as the ones in Figures~\ref{figbasicfrieze} and \ref{figexfrieze5p1}.

\begin{defin}\label{defperiodic}
An infinite frieze $(m_{ij})_{j-i\ge -2}$ is called $n$-\emph{periodic} and denoted by $\mathcal{F}_n$ if there exists an integer $n\geq 1$ such that $m_{ij}=m_{i+n,j+n}$ for all $i\le j $. A \emph{fundamental domain} $\mathcal{D}$ for $\mathcal{F}_n$ is given by $n$ consecutive {\sc se}-diagonals of $\mathcal{F}_n$.
\end{defin}

Note that, if the quiddity row of an infinite frieze is periodic, the whole infinite frieze is periodic. Every non-trivial row of a periodic infinite frieze is given by a repeating sequence of positive integers, up to cyclic equivalence. Therefore, the entire periodic infinite frieze is covered by a fundamental domain by successive copies in horizontal direction, thus a fundamental domain contains all the information about the periodic infinite frieze.

\begin{defin}\label{defperiodicqs}
Given an $n$-periodic infinite frieze $\mathcal{F}_n=(m_{ij})_{j-i\ge -2}$, its \emph{quiddity sequence} $q_{\mathcal{F}_n}$ is the $n$-tuple $(a_1,a_2, \dots ,a_{n})$, determined up to cyclic equivalence, with $a_i=m_{ii}$ for all $i\in\{1,2\dots, n\}$.
\end{defin}

Given an infinite frieze $\mathcal{F}=(m_{ij})_{j-i\ge -2}$, we use the following notions below. For a fixed $i\in \mathbb{Z}$ we denote by the \emph{{\sc se}-diagonal} $f(a_i)$ through $a_i$ the infinite sequence $(m_{ij})_{j\ge i-2}$. We will usually drop $a_i$ whenever it is clear from the context. Similarly, $\bar f(a_j)=(m_{ij})_{i\le j+2}$ is the \emph{{\sc sw}-diagonal} of $\mathcal{F}$ through $a_j$. The {\sc se} ({\sc sw}) sign stands for south-east (south-west), the direction of the diagonal.

Clearly, the entry $m_{ij}$ is the intersection of the {\sc se}-diagonal through $a_i$ and the {\sc sw}-diagonal through $a_j$ (see Figure~\ref{figlemrelqsd}).

\begin{rem}
Given a {\sc se}-diagonal of an infinite frieze the unimodular rule enables us to fill the next {\sc se}-diagonal to the right (east), starting at the top. The analogous result is true for {\sc sw}-diagonals. Hence a fundamental domain of a periodic infinite frieze, and thus the periodic infinite frieze itself, is also determined as soon as one {\sc se}-diagonal or one {\sc sw}-diagonal, respectively, is given. If an infinite frieze is not periodic, a {\sc se}-diagonal and a {\sc sw}-diagonal with a common entry different from zero are needed to determine the whole infinite frieze.
\end{rem}

\begin{figure}[H]
\scalebox{1}{\begin{tikzpicture}[font=\normalsize] 
  \matrix(m) [matrix of math nodes,row sep={1.6em,between origins},column sep={1.6em,between origins},nodes in empty cells]{
m_{i,i-2}&&0&&0&&\node{\cdots};&&&&\node{\cdots};&&0&&0&&m_{j+2,j}\\
&m_{i,i-1}&&1&&1&&\node{\cdots};&&\node{\cdots};&&1&&1&&m_{j+1,j}&\\
&&a_i&&a_{i+1}&&a_{i+2}&&\node{\cdots};&&a_{j-2}&&a_{j-1}&&a_j&&\\
&&&\node[rotate=-6.5,shift={(-0.034cm,-0.08cm)}]  {\ddots};&&\node[rotate=-6.5,shift={(-0.034cm,-0.08cm)}]  {\ddots};&&\node[rotate=-6.5,shift={(-0.034cm,-0.08cm)}]  {\ddots};&&\node[rotate=6.5,shift={(-0.034cm,-0.08cm)}]  {\iddots};&&\node[rotate=6.5,shift={(-0.034cm,-0.08cm)}]  {\iddots};&&\node[rotate=6.5,shift={(-0.034cm,-0.08cm)}]  {\iddots};&&&\\
&&&&&&&&&&&&&&&&\\
&&&&&\node[rotate=-6.5,shift={(-0.034cm,-0.08cm)}]  {\ddots};&&\node[rotate=6.5,shift={(-0.034cm,-0.08cm)}]  {\iddots};&&\node[rotate=-6.5,shift={(-0.034cm,-0.08cm)}]  {\ddots};&&\node[rotate=6.5,shift={(-0.034cm,-0.08cm)}]  {\iddots};&&&&&\\
&&&&&&m_{i,j-2}&&&&m_{i+2,j}&&&&&&\\
&&&&&&&m_{i,j-1}&&m_{i+1,j}&&&&&&&\\
&&&&&&&&m_{ij}&&&&&&&&\\
     };
\end{tikzpicture}}
\caption{A cone in an infinite frieze with apex at $m_{ij}$ framed by a pair of intersecting diagonals, namely the \mbox{{\sc se}-diagonal} $f(a_i)$ and the \mbox{{\sc sw}-diagonal} $\bar f(a_j)$.}\label{figlemrelqsd}
\end{figure}

Motivated by the work of Conway and Coxeter in \cite{CCI} the next lemma describes how the entries of an infinite frieze and its quiddity row depend on each other.

\begin{lem}\label{lemrelationqsdiag}
Let $(a_i)_{i\in\mathbb{Z}}$ be the quiddity row of an infinite frieze $\mathcal{F}=(m_{ij})_{j-i\ge -2}$. Then, for all $j \geq i$,
\begin{enumerate}[$a)$]
\item\label{relqsda}  $\displaystyle a_j=\frac{m_{ij}+m_{i,j-2}}{m_{i,j-1}}$ \quad and \quad $\displaystyle a_i=\frac{m_{ij}+m_{i+2,j}}{m_{i+1,j}}$,\\[1em]
\item\label{relqsdb} $\displaystyle m_{ij}=\det \begin{pmatrix}a_i&1&&&0\\
1&a_{i+1}&1&\\
&\ddots&\ddots&\ddots\\
&&1&a_{j-1}&1\\
0&&&1&a_j
\end{pmatrix}$.
\end{enumerate}
\end{lem}

\begin{proof}
\begin{inparaenum}[$a)$] 
\item Clearly, the first equality is true for $j=i$. Now, we continue similar to the proof of the analogue relation valid for finite friezes. For $j\geq i$ we use that by the unimodular rule the entries in any two {\sc se}-diagonals $f,f'$ of $\mathcal{F}$ arranged as in the following figure on the right satisfy the equality on the left, where by definition, $x_2$ and $x_2'$ are positive integers.

{\hfil \begin{tikzpicture}[font=\normalsize] 
  \matrix(m) [matrix of math nodes,row sep={1.25em,between origins},column sep={1.25em,between origins},nodes in empty cells]{
\node[gray]  {f};&&&&\node[gray]  {f'};&&&&&&&&&&&&&&\\
&\node {0};&&&&0&&&&&&&&&&&&&\\
&&\node{1};&&&&1&&&&&&&&&&&&\\
&&&&&&&\node[rotate=-6.5,shift={(-0.034cm,-0.08cm)}]  {\ddots};&&&&&&&&&&&\\
&&&&\node[rotate=-6.5,shift={(-0.034cm,-0.08cm)}]  {\ddots};&&&& x'_1&&&&&&&&&&\displaystyle{\frac{x_1+x_3}{x_2}=\frac{x'_1+x'_3}{x'_2}}\\
&&&&&&&\node[rotate=6.5,shift={(-0.034cm,-0.08cm)}]  {\iddots};&&x'_2&&&&&&&&&\\
&&&&&&\node {x_1};&&\node[rotate=6.5,shift={(-0.034cm,-0.08cm)}]  {\iddots};&&x'_{3}&&&&&&&&&&\\
&&&&&&&\node  {x_2};&&\node[rotate=6.5,shift={(-0.034cm,-0.08cm)}]  {\iddots};&&\node[rotate=-6.5,shift={(-0.034cm,-0.08cm)}]  {\ddots};&&&&&&&\\
&&&&&&&&\node  {x_3};&&&&&&&&&&\\
&&&&&&&&&\node[rotate=-6.5,shift={(-0.034cm,-0.08cm)}]  {\ddots};&&&&&&&&&\\
};
\end{tikzpicture}}

\noindent  Applying this to $x_1'=0,$ $x _2'=1,$ $x_3'=a_j$ and the corresponding entries in $f(a_i)$ with $x_1=m_{i,j-2},$ $x_2=m_{i,j-1},$ $x_3=m_{ij}$, as illustrated in Figure~\ref{figlemrelqsd}, we get the desired result. Similarly, one can check that the second equality holds.

\item The claim follows immediately by induction on $i$ and with $a)$.
\end{inparaenum}
\end{proof}

\begin{rem}\label{remrelationqsdiagperiodic}
In particular, for an $n$-periodic infinite frieze $\mathcal{F}_n=(m_{ij})_{j-i\ge -2}$ with quiddity sequence $q_{\mathcal{F}_n}=(a_1,a_2, \dots, a_{n})$, the quiddity row $(a_i)_{i\in\mathbb{Z}}$ of $\mathcal{F}_n$ satisfies $a_i=a_{i+l n}$, for all $i\in\{ 1,\dots, n\}$ and every integer $l$.
Then Lemma~\ref{lemrelationqsdiag} implies
$$a_j=\frac{m_{i,j-2+ln}+m_{i,j+ln}}{m_{i,j-1+ln}}$$
for all $j\in\{1,\dots, n\}$ and every non-negative integer $l$.
Moreover,
for all $j\geq i$, we have
$$m_{ij}=\det \begin{pmatrix}
d_0&1&&&0\\
1&d_1&1&&\\
&\ddots&\ddots&\ddots&\\
&&1&d_{j-i-1}&1\\
0&&&1&d_{j-i}
\end{pmatrix},$$
where $d_{k}=a_{i+k \mod n}$ for $k\in\{0,1,\dots,j-i\}$.
\end{rem}

Note that by Lemma~\ref{lemrelationqsdiag} \ref{relqsda}) for three consecutive diagonals of an infinite frieze the middle one is proportional to the sum of its two neighbors. This result can be formulated for {\sc se}-diagonals as follows. Clearly, there is an analogue for {\sc sw}-diagonals.

\begin{cor}\label{cordiagonalsum}
Let $\mathcal{F}=(m_{ij})_{j-i\ge -2}$ be an infinite frieze with quiddity row $q_{\mathcal{F}}=(a_i)_{i\in \mathbb{Z}}$ and $k$ be a fixed integer. Then the {\sc se}-diagonal $f(a_k)$ is given by the two neighboring \mbox{{\sc se}-diagonals} $f(a_{k-1})$ and $f(a_{k+1})$ via
$$a_{k-1}\cdot m_{kj}= m_{k-1,j}+ m_{k+1,j}$$
for all $j\geq k-1$. Moreover, if $a_{k-1}=1$, then $f(a_k)$ is given by the \enquote{shifted sum} of $f(a_{k-1})$ and $f(a_{k+1})$.
\end{cor}

Before studying operations on infinite friezes in the next section we have a short look at \mbox{$1$-periodic} infinite friezes. It makes sense to call them \emph{complete infinite friezes} since they can arise from complete graphs. It is an easy induction argument using Corollary~\ref{cordiagonalsum} to prove that complete infinite friezes are determined as follows.
\pagebreak
\begin{prop}\label{propcompletefrieze}
Let $a\geq 2$ be an integer and let $q=(a_i)_{i\in\mathbb{Z}}$ be the constant sequence with $a_i=a$ for all $i\in\mathbb{Z}$. Then $q$ is the quiddity row of a complete infinite frieze $\mathcal{F}_a$. Moreover, the entries $m_{ij},j-i\ge -2$, of $\mathcal{F}_a$ are given by
$$m_{ij}=\sum\limits_{k=0}^{\big\lfloor \frac {j-i+1}2 \big\rfloor} (-1)^k \binom{j-i+1-k}{k} a^{j-i+1-2k}.$$
\end{prop}

%
\section{Cutting and gluing infinite friezes}\label{seccutglue}
%

In this section we point out two particular ways how an infinite frieze can be modified to obtain a new infinite frieze by extending the work of  Conway and Coxeter in \cite{CCI}. We also describe how these algebraic operations can be used on periodic infinite friezes.

\subsection{Gluing}\label{secglue}

The first operation that we describe on infinite friezes produces a new infinite frieze starting from an infinite frieze by inserting a pair of diagonals. Note that this is not the only possibility to define an operation on infinite friezes that enlarges them.

\begin{thm}\label{thmgluefrieze}
Let $\mathcal{F}$ be an infinite frieze with quiddity row $(a_i)_{i\in\mathbb{Z}}$ and $k$ be a fixed integer. Then the sequence $(\hat a_i)_{i\in\mathbb{Z}}$ defined by
$$\hat a_i=\begin{cases}
a_i&\text{if }i\leq k-1,\\
a_k\!+\!1&\text{if }i=k,\\
1&\text{if }i=k+1,\\
a_{k+1}\!+\!1&\text{if }i=k+2,\\
a_{i-1}&\text{if }i\geq k+3,
\end{cases}$$
is the quiddity row of an infinite frieze $\widehat{\mathcal{F}}$.
\end{thm}

For an infinite frieze $\mathcal{F}$, we call the operation that maps its quiddity row $(a_i)_{i\in\mathbb{Z}}$ to $(\hat a_i)_{i\in\mathbb{Z}}$ \emph{gluing above the pair $(a_k,a_{k+1})$}, where $(\hat a_i)_{i\in\mathbb{Z}}$ is as in Theorem~\ref{thmgluefrieze}. The situation is illustrated in Figure~\ref{figgluefrieze}.

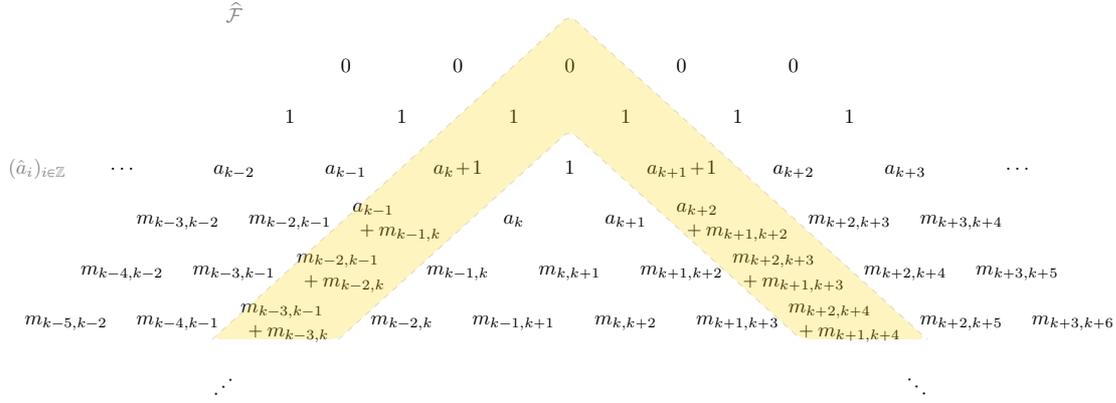
\begin{figure}[t]
\resizebox{\linewidth}{!}{\begin{tikzpicture}[font=\normalsize] 
  \matrix(m) [matrix of math nodes,row sep={2.5em,between origins},column sep={2.75em,between origins},nodes in empty cells]{
&&&\node[gray]  {\widehat{\mathcal{F}}};&&&&&&&&&&&&&&&\\
&&&&&0&&0&&0&&0&&0&&&&&\\
&&&&1&&1&&1&&1&&1&&1&&&&\\
\node[gray,shift={(-0.5cm,0cm)}]  {(\hat a_i)_{i\in\mathbb{Z}}};&\node{\cdots};&&a_{k-2}&&a_{k-1}&&a_k\!+\!1&&1&&a_{k+1}\!+\!1&&a_{k+2}&&a_{k+3}&&\node{\cdots};&\\
&&m_{k-3,k-2}&&m_{k-2,k-1}&&\begin{array}{@{}l} a_{k-1}\\\enspace\!+\,m_{k-1,k}\end{array}&&a_k&&a_{k+1}&&\begin{array}{@{}l} a_{k+2} \\\enspace+\, m_{k+1,k+2}\end{array}&&m_{k+2,k+3}&&m_{k+3,k+4}&&\\
&m_{k-4,k-2}&&m_{k-3,k-1}&&\begin{array}{@{}l} m_{k-2,k-1}\\\enspace\!+\,m_{k-2,k}\end{array}&&m_{k-1,k}&&m_{k,k+1}&&m_{k+1,k+2}&&\begin{array}{@{}l} m_{k+2,k+3} \\\enspace+\,m_{k+1,k+3}\end{array}&&m_{k+2,k+4}&&m_{k+3,k+5}&\\
m_{k-5,k-2}&&m_{k-4,k-1}&&\begin{array}{@{}l} m_{k-3,k-1}\\\enspace\!+\,m_{k-3,k}\end{array}&&m_{k-2,k}&&m_{k-1,k+1}&&m_{k,k+2}&&m_{k+1,k+3}&&\begin{array}{@{}l} m_{k+2,k+4} \\\enspace+\,m_{k+1,k+4}\end{array}&&m_{k+2,k+5}&&m_{k+3,k+6}\\
&&&\node[rotate=6.5,shift={(-0.216cm,-0.33cm)}]  {\iddots};&&&&&&&&&&&&\node[rotate=-6.5,shift={(0.216cm,-0.33cm)}]  {\ddots};&&&\\
};

\draw[rounded corners,semithick,gray,dashed,opacity=0.25] ($(m-7-4.west)+(-0.25,-0.25)$)--(m-1-10.north) -- ($(m-7-16.east)+(0.25,-0.25)$);

\draw[rounded corners,semithick,gray,dashed,opacity=0.25] ($(m-7-6.east)+(-0.25,-0.25)$)--(m-3-10.south) -- ($(m-7-14.west)+(0.25,-0.25)$);

\draw[rounded corners,opacity=0,fill=myyellow,fill opacity=0.25]  (m-1-10.north) -- ($(m-7-4.west)+(-0.25,-0.25)$)--($(m-7-6.east)+(-0.25,-0.25)$) -- (m-3-10.south) -- ($(m-7-14.west) +(0.25,-0.25)$)-- ($(m-7-16.east)+(0.25,-0.25)$)--cycle;
\end{tikzpicture}}
\caption{The infinite frieze $\widehat{\mathcal{F}}$ produced by gluing above the pair $(a_k,a_{k+1})$ in an infinite frieze $\mathcal{F}=(m_{ij})_{j-i\ge -2}$.}\label{figgluefrieze}
\end{figure}

\begin{proof}
Let $\mathcal{F}=(m_{ij})_{j-i\ge -2}$ be an infinite frieze with quiddity row $(a_i)_{i\in\mathbb{Z}}$. We choose $k\in\mathbb{Z}$ and consider the array $\widehat{\mathcal{F}}=(\hat m_{ij})_{j-i\ge -2}$ with $\hat m_{i,i-2}=0,\hat m_{i,i-1}=1$ and, for $j-i\ge0$, given by the entries of $\mathcal{F}$ as shown in Figure~\ref{figgluefrieze}
$$
\hat m_{ij}=\begin{cases}
m_{ij}&\text{if }j\le k-1,\\
m_{i,j-1}+m_{ij} &\text{if } j=k,\\
m_{i,j-1} &\text{if }i\le k+1\le j,\\
m_{i,j-1}+m_{i-1,j-1} &\text{if }i=k+2,\\
 m_{i-1,j-1}&\text{if }i\ge k+3
\end{cases}
$$
We show that $\widehat{\mathcal{F}}$ is an infinite frieze. Clearly, since $\mathcal{F}$ is an infinite frieze, we have $\hat m_{ij}>0$ if $j-i\ge0$. It remains to show that the unimodular rule is satisfied for every diamond in $\widehat{\mathcal{F}}$, i.e.\ $\hat m_{ij}\hat m_{i+1,j+1}-\hat m_{i+1,j}\hat m_{i,j+1}=1$ for $j-i\geq -1$. In particular, for $j-i= -1$, this is always true, so we may assume that $j-i\geq 0$.

As long none of the entries in the yellow diagonals are involved, namely the {\sc sw}-diagonal $\bar f(a_k)$ and the {\sc se}-diagonal $f(a_{k+2})$ (see Figure~\ref{figgluefrieze}), this property is immediately inherited by $\mathcal{F}$. Otherwise, a diamond involves two entries either contained in $\bar f(a_k)$, or in $f(a_{k+2})$ and we have to check two separate cases for each diagonal. Firstly, if the diamond lies to the left of $\bar f(a_k)$, we have
\begin{align*}
\hat m_{i,k-1}\hat m_{i+1,k}-\hat m_{i+1,k-1}\hat m_{ik}&=m_{i,k-1}\left(m_{i+1,k-1}+m_{i+1,k}\right) -m_{i+1,k-1}\left(m_{i,k-1}+m_{ik}\right)\\
&=m_{i,k-1}m_{i+1,k}-m_{i+1,k-1}m_{ik}=1
\end{align*}
for all  $i\le k-1$. Similarly, if the diamond lies to the right of $\bar f(a_k)$, we get $\hat m_{i,k}\hat m_{i+1,k+1}-\hat m_{i+1,k}\hat m_{i,k+1}=1$ for all $i\le k $. The two cases for $f(a_{k+2})$ are completely analogous. Hence $\widehat{\mathcal{F}}$ is an infinite frieze with quiddity row \mbox{$(\dots, a_{k-2},a_{k-1} ,a_k\!+\!1, 1, a_{k+1}\!+\!1,a_{k+2},a_{k+3}, \dots)$}, as desired.
\end{proof}

\begin{rem}
In the proof of the theorem the effect of gluing into the initial infinite frieze is given explicitly: $\widehat{\mathcal{F}}$ is obtained from $\mathcal{F}$ by inserting simultaneously a pair of diagonals. In particular, a {\sc se}-diagonal and a {\sc sw}-diagonal with common entry in the first row of $\widehat{\mathcal{F}}$ are inserted such that  every entry of the new diagonals in $\widehat{\mathcal{F}}$ is given by the sum of the two closest entries in the neighboring diagonals to the left and to the right, see Figure~\ref{figgluefrieze}. That is precisely what the second part of Corollary~\ref{cordiagonalsum} says. The \mbox{symbol\, $\widehat{\ }$} denotes that we glued in a pair of diagonals.
\end{rem}

Note that if we start with the basic frieze $\mathcal{F}_{\ast}$ given in Figure~\ref{figbasicfrieze}, then gluing serves as a tool to produce new infinite friezes from $\mathcal{F}_{\ast}$. The next result follows immediately.

\begin{cor}
There exist infinitely many infinite friezes.
\end{cor} 

Clearly, as soon as we consider periodic infinite friezes we lose the periodicity after applying the operation of gluing once. To remedy this we define a slightly different operation that preserves periodicity of an infinite frieze using the current set up of gluing.

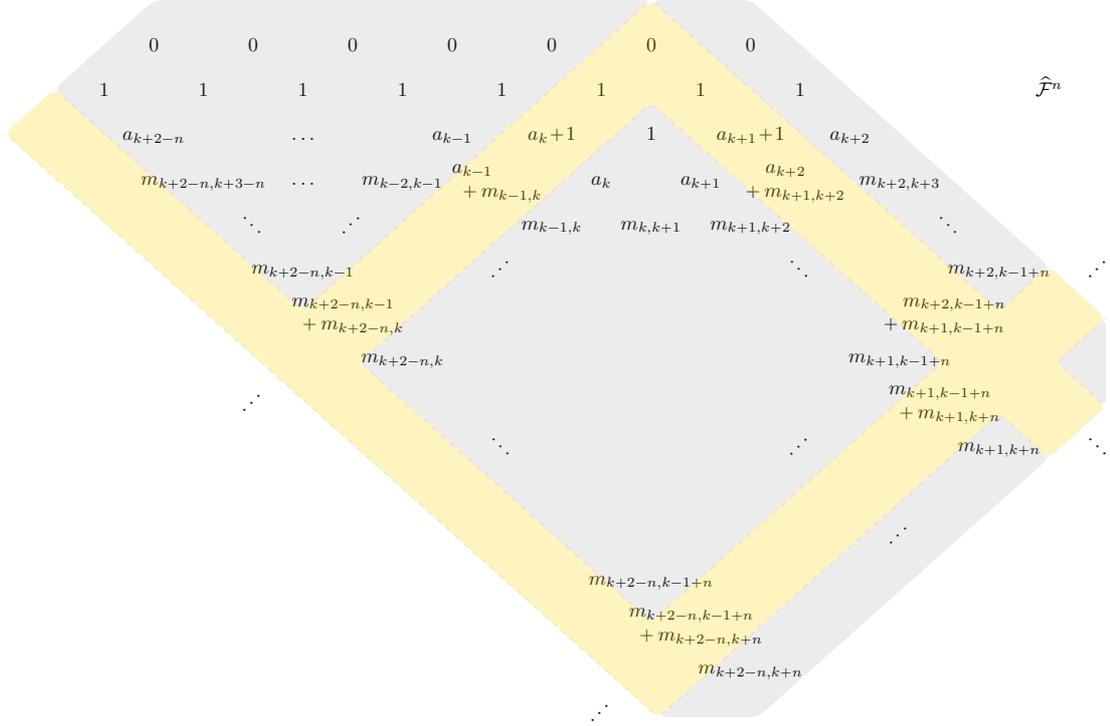
\begin{figure}[t]
\resizebox{\linewidth}{!}{\begin{tikzpicture}[font=\normalsize] 
  \matrix(m) [matrix of math nodes,row sep={2.25em,between origins},column sep={2.5em,between origins},nodes in empty cells]{
&&&&&&&&&&&&&&&&&&&&&&\\
&&&0&&0&&0&&0&&0&&0&&0&&&&&&&\\
&&1&&1&&1&&1&&1&&1&&1&&1&&&&&\widehat{\mathcal{F}}^n&\\
&&&a_{k+2-n}&&&\dots&&&a_{k-1}&&a_k\! +\!1&&1&&a_{k+1}\!+\!1&&a_{k+2}&&&&&\\
&&&&m_{k+2-n,k+3-n}&&\dots&&m_{k-2,k-1}&&\begin{array}{@{}l}a_{k-1}\\\enspace+\,m_{k-1,k}\end{array}&&a_k&&a_{k+1}&&\begin{array}{@{}l}\quad a_{k+2}\\+\,m_{k+1,k+2}\end{array}&&m_{k+2,k+3}&&&&\\
&&&&&\node[rotate=-6.5,shift={(-0.034cm,-0.08cm)}]  {\ddots};&&\node[rotate=6.5,shift={(-0.034cm,-0.08cm)}]  {\iddots};&&&&m_{k-1,k}&&m_{k,k+1}&&m_{k+1,k+2}&&&&\node[rotate=-6.5,shift={(-0.034cm,-0.08cm)}]  {\ddots};&&&\\
&&&&&&m_{k+2-n,k-1}&&&&\node[rotate=6.5,shift={(-0.034cm,-0.08cm)}]  {\iddots};&&&&&&\node[rotate=-6.5,shift={(-0.034cm,-0.08cm)}]  {\ddots};&&&&m_{k+2,k-1+n}&&\node[rotate=6.5,shift={(-0.034cm,-0.08cm)}]  {\iddots};\\
&&&&&&&\begin{array}{@{}l} m_{k+2-n,k-1}\\\enspace+\,m_{k+2-n,k}\end{array}&&&&&&&&&&&&\begin{array}{@{}l}\quad m_{k+2,k-1+n}\\+\,m_{k+1,k-1+n}\end{array}&&&\\
&&&&&&&&m_{k+2-n,k}&&&&&&&&&&m_{k+1,k-1+n}&&&&\\
&&&&&\node[rotate=6.5,shift={(-0.034cm,-0.08cm)}]  {\iddots};&&&&&&&&&&&&&&\begin{array}{@{}l}m_{k+1,k-1+n}\\\enspace+\,m_{k+1,k+n}\end{array}&&&\\
&&&&&&&&&&\node[rotate=-6.5,shift={(-0.034cm,-0.08cm)}]  {\ddots};&&&&&&\node[rotate=6.5,shift={(-0.034cm,-0.08cm)}]  {\iddots};&&&&m_{k+1,k+n}&&\node[rotate=-6.5,shift={(-0.034cm,-0.08cm)}]  {\ddots};\\
&&&&&&&&&&&&&&&&&&&&&&\\
&&&&&&&&&&&&&&&&&&\node[rotate=6.5,shift={(-0.034cm,-0.08cm)}]  {\iddots};&&&&\\
&&&&&&&&&&&&&m_{k+2-n,k-1+n}&&&&&&&&&\\
&&&&&&&&&&&&&&\begin{array}{@{}l}m_{k+2-n,k-1+n}\\\enspace+\,m_{k+2-n,k+n}\end{array}&&&&&&&&\\
&&&&&&&&&&&&&&&m_{k+2-n,k+n}&&&&&&&\\
&&&&&&&&&&&&\node[rotate=6.5,shift={(-0.034cm,-0.08cm)}]  {\iddots};&&&&&&&&&&\\
};

\draw[rounded corners,opacity=0,fill=myyellow,fill opacity=0.25](m-4-1.north)-- (m-3-2.north)--(m-8-7.north)--(m-1-14.north)  -- (m-8-21.north)--(m-7-22.north) --(m-8-23.east) --(m-9-22.east)-- (m-10-23.east)-- (m-11-22.south) --(m-10-21.south) -- (m-17-14.east)--(m-16-13.north) --(m-9-20.west) -- (m-3-14.south)  -- (m-9-8.east)--(m-16-15.north) -- (m-17-14.east) -- cycle;

\draw[rounded corners,opacity=0,fill=gray,fill opacity=0.15]  (m-9-8.east)--  (m-3-14.south) -- (m-9-20.west)-- (m-15-14.north) -- cycle; 
\draw[rounded corners,opacity=0,fill=gray,fill opacity=0.15]  (m-3-2.north)--(m-8-7.north)--(m-1-14.north) --  (m-1-4.north)--cycle; 
\draw[rounded corners,opacity=0,fill=gray,fill opacity=0.15] (m-1-14.north) --(m-8-21.north)--(m-7-22.north)--(m-1-16.north)-- cycle;
 \draw[rounded corners,opacity=0,fill=gray,fill opacity=0.15]  (m-8-23.east) --(m-9-22.east)-- (m-10-23.east)--cycle;
\draw[rounded corners,opacity=0,fill=gray,fill opacity=0.15] (m-17-14.east)--(m-17-16.east)-- (m-11-22.south) -- (m-10-21.south) --cycle; 

\draw[semithick,gray,dashed,rounded corners,opacity=0.15]  (m-3-2.north)--(m-8-7.north)-- (m-1-14.north)  -- (m-8-21.north)--(m-7-22.north);
\draw[semithick,gray,dashed,rounded corners,opacity=0.15]  (m-11-22.south) --(m-10-21.south) -- (m-17-14.east)--(m-4-1.north);
\draw[semithick,gray,dashed,rounded corners,opacity=0.15]  (m-8-23.east) --(m-9-22.east)-- (m-10-23.east);
\draw[semithick,gray,dashed,rounded corners,opacity=0.15] (m-15-14.north) --(m-9-20.west) -- (m-3-14.south)  -- (m-9-8.east)--cycle;

\end{tikzpicture}}
\caption{The $(n\!+\!1)$-periodic infinite frieze $\widehat{\mathcal{F}}^n$ obtained by \mbox{$n$-gluing} above $(a_k,a_{k+1})$ in an $n$-periodic infinite frieze $\mathcal{F}$ colored light grey.}\label{fignglueperiodicfrieze}
\end{figure}

\begin{prop}\label{propnglueperiodicfrieze}
Given an $n$-periodic infinite frieze $\mathcal{F}$ with quiddity sequence \mbox{$ q_{\mathcal{F}}=(a_1,a_2, \dots , a_{n})$} and let $k\in\{1,2,\dots,n\}$ ba an integer. Then the $(n\!+\!1)$-tuple
$$\hat{q}_{\mathcal{F}}^n=
\begin{cases}
(a_1\!+\!2,1) &\text{if } n=1,\\
(a_1,\dots,a_{k-1},a_k\!+\!1,1,a_{k+1}\!+\!1,a_{k+2},\dots , a_{n}) &\text{otherwise},
\end{cases}
$$
where indices are taken modulo $n$, leads to an infinite frieze $\widehat{\mathcal{F}}^n$ of period $n\!+\!1$.
\end{prop}

For an $n$-periodic infinite frieze $\mathcal{F}$ with quiddity sequence $q_{\mathcal{F}}$ the operation mapping  $\mathcal{F}$ to the $(n\!+\!1)$-periodic infinite frieze $\widehat{\mathcal{F}}^n$ with quiddity sequence $\hat{q}_{\mathcal{F}}^n$, defined in Proposition~\ref{propnglueperiodicfrieze}, is called \emph{$n$-gluing above the pair $(a_k,a_{k+1})$}.

\begin{proof}
Let $(a_i)_{i\in\mathbb{Z}}$ be the quiddity row of an $n$-periodic frieze $\mathcal{F}$. We choose $k\in\{1,2,\dots,n\}$. Then for every integer $l$, we glue above the pair $(a_{k+ln},a_{k+1+ln})$ as defined in Theorem~\ref{thmgluefrieze} and obtain an infinite frieze, denoted by $\widehat{\mathcal{F}}^n$. By construction, the quiddity row of $\widehat{\mathcal{F}}^n$ is $n\!+\!1$ periodic and determined, up to cyclic permutation, by the $2$-tuple $(a_1\!+\!2,1)$, if $n=1$, or by the $(n\!+\!1)$-tuple $(a_1, \dots ,a_{k-1},a_k\!+\!1,1,a_{k+1}\!+\!1,a_{j+2},\dots , a_{n})$ otherwise. Hence the result.
\end{proof}

Note that the operation of $n$-gluing may provide an infinite frieze of period dividing $n\!+\!1$ but strictly smaller than $n\,+\,1$. Note also that $n$-gluing only depends on the pair $(a_k,a_{k+1})$ of integers not on the choice of the quiddity sequence.

Clearly, there are other operations on periodic infinite friezes preserving periodicity. E.g.\ given a quiddity sequence $(a_1,a_2, \dots , a_{n})$ of an $n$-periodic infinite frieze the sequence $(a_1, \dots , a_k\!+~\!1,1,\linebreak a_{k+1}\!+~\!1, \dots , a_{n},a_1,\dots,a_n)$ also leads to a periodic infinite frieze, namely of period $2n\!+\!1$, and so on.

The next corollary describes the modification to the entries of the initial periodic infinite friezes caused by $n$-gluing, see Figure~\ref{fignglueperiodicfrieze}. It we will be convenient to use the following notation: for $i,x\in\mathbb{Z}$, we set $i_x=i-t$ whenever $x+(t-1)(n+1)< i\le x+t(n+1)$.
 
\begin{cor}\label{cornglueperiodicfrieze}
Let $\widehat{\mathcal{F}}^n=(\hat m_{ij})_{j-i\ge -2}$ be the $(n\!+\!1)$-periodic infinite frieze obtained from an\linebreak $n$-periodic infinite frieze $\mathcal{F}=(m_{ij})_{j-i\ge -2}$ with quiddity sequence $q_{\mathcal{F}}=(a_1,a_2, \dots ,a_{n})$ by $n$-gluing above the pair $(a_k,a_{k+1})$, for some $k\in\{1,2,\dots,n\}$. Then
$$ \hat m_{ij}=\begin{cases}
m_{i_{k+2}j_k} &\text{if } i\not \equiv k+2, j \not \equiv k,\\
m_{i_{k+2},j_k-1}+m_{i_{k+2}j_k} &\text{if } i\not \equiv k+2, j \equiv k,\\
m_{i_{k+2}-1,j_k}+m_{i_{k+2}j_k} &\text{if } i \equiv k+2,j \not \equiv k,\\
m_{i_{k+2}-1,j_k-1}+m_{i_{k+2}j_k}+m_{i_{k+2}-1,j_k}+m_{i_{k+2},j_k-1} &\text{if } i \equiv k+2, j \equiv k,
\end{cases}$$
where $\equiv$ means equal reduced modulo $n+1$.
\end{cor}

Before considering the reverse operation to gluing we illustrate the operation $n$-gluing with an example. The example points out that $n$-gluing preserves more than the fist two rows of the initial friezes. Diamond-shaped fragments of the old friezes appear in the new frieze. Moreover, the example already illustrates how the reverse operation to $n$-gluing will work.

\begin{ex}\label{ex1}
Let $n=3$ and consider the basic infinite frieze $\mathcal{F}_\ast$ given in Figure~\ref{figbasicfrieze} with quiddity sequence $q_{\mathcal{F}_\ast}=(2,2,2)$. Now we perform a $3$-gluing above the pair $(a_2,a_3)$. This leads to the new quiddity sequence $\hat{q}_{\mathcal{F}_\ast}^3=(2,3,1,3)$ determining the infinite frieze $\widehat{\mathcal{F}}_\ast^3$ of period $4$, as shown in the figure below. The inserted pairs of diagonals are colored yellow. In particular, outside these diagonals $\widehat{\mathcal{F}}_\ast^3$ coincides with $\mathcal{F}_\ast$.

{\hfil \scalebox{.9}
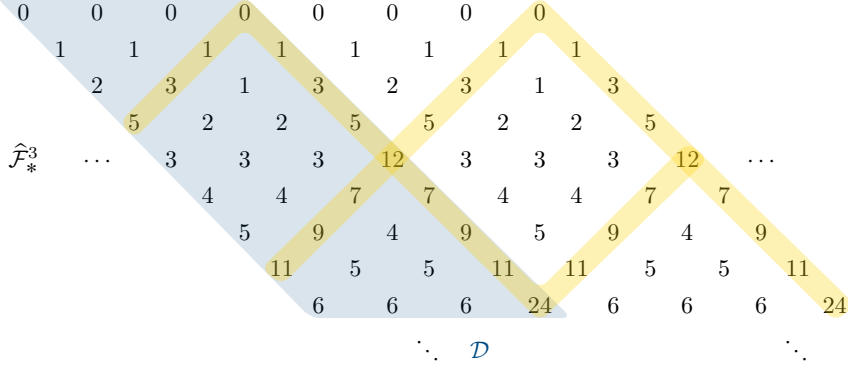
{\begin{tikzpicture}[font=\normalsize] 
  \matrix(m) [matrix of math nodes,row sep={1.5em,between origins},column sep={1.5em,between origins},nodes in empty cells]{
0&&0&&0&&0&&0&&0&&0&&0&&&&&&&&\\
&1&&1&&1&&1&&1&&1&&1&&1&&&&&&&\\
&&2&&3&&1&&3&&2&&3&&1&&3&&&&&&\\
&&&5&&2&&2&&5&&5&&2&&2&&5&&&&&\\
\widehat{\mathcal{F}}_\ast^3&&\node{\cdots};&&3&&3&&3&&12&&3&&3&&3&&12&&\node{\cdots};&&\\
&&&&&4&&4&&7&&7&&4&&4&&7&&7&&&\\
&&&&&&5&&9&&4&&9&&5&&9&&4&&9&&\\
&&&&&&&11&&5&&5&&11&&11&&5&&5&&11&\\
&&&&&&&&6&&6&&6&&24&&6&&6&&6&&24\\[1mm]
&&&&&&&&&&&\node[rotate=-6.5,shift={(-0.034cm,-0.08cm)}]  {\ddots};&&&&&&&&&&\node[rotate=-6.5,shift={(-0.034cm,-0.08cm)}]  {\ddots};&\\
};

\draw (m-10-13) node[myblue,shift={(0.2cm,0.1cm)}]{$\mathcal{D}$};
\draw[opacity=0,rounded corners,fill=myblue,fill opacity=0.15] ($(m-1-1.north west)+(-0.2cm,0cm)$) -- ($(m-9-9.south west)+(0.1cm,0.05cm)$) -- ($(m-9-16.south west)+(0.05cm,0.05cm)$) -- (m-1-7.north) -- cycle;

\draw[rounded corners,opacity=0,fill=myyellow,fill opacity=0.3]  (m-4-4.west)-- (m-1-7.north) -- (m-9-15.north)-- (m-5-19.north) -- (m-5-19.east)-- (m-9-15.south) -- (m-1-7.south) -- (m-4-4.south)-- cycle;

\draw[rounded corners,opacity=0,fill=myyellow,fill opacity=0.3] (m-8-8.west) -- (m-1-15.north) -- (m-9-23.east)--(m-9-23.south)-- (m-1-15.south) -- (m-8-8.south) -- cycle;

\end{tikzpicture}}}
\end{ex}

\subsection{Cutting}\label{seccut}

The second operation on infinite friezes we define produces a new infinite frieze starting from a given infinite frieze whenever an entry $1$ appears in its quiddity row. The analogue for this in the case of tame $\mathrm{SL}_2$-tilings was considered in \cite[Lemma 5]{BR}.

\begin{thm}\label{thmcutfrieze}
Let $\mathcal{F}$ be an infinite frieze with quiddity row $(a_i)_{i\in\mathbb{Z}}$ such that $a_k=1$ for some integer $k$. Then  the sequence $(\check a_i)_{i\in\mathbb{Z}}$ with
$$\check a_i=\begin{cases}
a_i&\text{if }i\leq k-2,\\
a_{k-1}\!-\!1&\text{if }i=k-1,\\
a_{k+1}\!-\!1&\text{if }i=k,\\
a_{i+1}&\text{if }i\geq k+1,
\end{cases}$$
is the quiddity row of an infinite frieze $\widecheck{\mathcal{F}}$.
\end{thm}

We say that that $\widecheck{\mathcal{F}}$ is obtained from $\mathcal{F}$ by \emph{cutting above $a_k=1$}, where $\widecheck{\mathcal{F}}$ is as given in Theorem~\ref{thmcutfrieze}. 

\begin{proof}
Let $\mathcal{F}=(m_{ij})_{j-i\ge -2}$ be an infinite frieze and let $(a_i)_{i\in\mathbb{Z}}$ be its quiddity row such that $a_k=1$ for some $k\in\mathbb{Z}.$ By the second part of Corollary~\ref{cordiagonalsum}, the {\sc se}-diagonal $f(a_{k+1})$ through $a_{k+1}$ is given by the shifted sum of $f(a_k)$ and $f(a_{k+2})$, that is $m_{k+1,j}= m_{kj}+ m_{k+2,j}.$ Analogously, for the {\sc sw}-diagonal $\bar f(a_{k-1})$ through $a_{k-1}$ we have $m_{i,k-1}= m_{i,k-2}+ m_{ik}.$ 

We now consider the array $\widecheck{\mathcal{F}}=(\check m_{ij})_{j-i\ge -2}$ obtained from $\mathcal{F}$ by suppressing $f(a_{k+1})$ and $\bar f(a_{k-1})$. For the labeling we fix the entries of $\mathcal{F}$ that are on the left of $\bar f(a_{k-1})$, in other words $\check m_{ij}=m_{ij}$ for all $,i\le k,j\le k-2$. Clearly, $\check m_{i,i-2}=0$ and since $\check m_{k,k-1}=a_k=1$ we have $\check m_{i,i-1}=1$ for all $i$. Moreover, we have $\check m_{ij}>0$ for all $j-i\ge 0.$ It remains to show the unimodular rule still holds along the two cutting lines, i.e.\ $m_{i,k-2}m_{i+1,k}-m_{i+1,k-2}m_{ik}=1$ and $m_{kj}m_{k+2,j+1}-m_{k+2,j}m_{k,j+1}=1$. Since $\mathcal{F}$ satisfies the unimodular rule and from above we get $1=m_{i,k-2}m_{i+1,k-1}-m_{i+1,k-2}m_{i,k-1}=m_{i,k-2}\left(m_{i+1,k-2}+ m_{i+1,k}\right)-m_{i+1,k-2}\left(m_{i,k-2}+ m_{ik}\right)=m_{i,k-2}m_{i+1,k}-m_{i+1,k-2}m_{ik}$ and $1=m_{kj}m_{k+1,j+1}-m_{k+1,j}m_{k,j+1}=m_{kj}\left( m_{k,j+1}+ m_{k+2,j+1}\right)-\left( m_{kj}+ m_{k+2,j}\right)m_{k,j+1}=m_{kj}m_{k+2,j+1}-m_{k+2,j}m_{k,j+1}.$ This completes the proof.
\end{proof}

\begin{rem}
In reverse to gluing, the effect caused on the initial infinite frieze by cutting is that a pair of diagonals is suppressed and the remaining entries stay unchained. Namely, the {\sc se}-diagonal through the right neighboring entry of the entry $1$ we cut above and the {\sc sw}-diagonal through the left neighboring entry of the same entry $1$. The symbol\, $\widecheck{\ }$\, indicates that two diagonals are removed. Clearly, first gluing above a pair and then cutting above the new entry that occurred in the quiddity sequence after gluing yields the original infinite frieze.
\end{rem}

Similarly, as in Section \ref{secglue}, the periodicity is lost if we consider periodic infinite friezes after cutting once. Thus, we define an operation on periodic infinite friezes that preserves periodicity.  Recall that the two neighbors of an entry $1$ in a given row are both strictly bigger than $1$. Thus, if a periodic infinite frieze has an entry $1$ in its quiddity sequence, the period has to be at least $2$.

\begin{prop}\label{propncutperiodicfrieze}
Given a quiddity sequence $q_{\mathcal{F}}=(a_1,a_2, \dots, a_{n})$ of an $n$-periodic infinite frieze $\mathcal{F}$ such that $a_k=1$ for some $k\in\{1,2,\dots, n\}$, the $(n\!-\!1)$-tuple
$$
\check{q}_{\mathcal{F}}^n=
\begin{cases}
(a_{k+1}\!-\!2) &\text{if } n=2,\\
(a_1, \dots ,a_{k-2},a_{k-1}\! -\!1,a_{k+1}\!-\!1,a_{k+2} ,\dots,  a_{n}) &\text{otherwise},
\end{cases}
$$
where indices are reduced modulo $n$, yields an infinite frieze $\widecheck{\mathcal{F}}^n$ of period $n\!-\!1$.
\end{prop}

We say that $\widecheck{\mathcal{F}}^n$ in Proposition~\ref{propncutperiodicfrieze} is obtained from $\mathcal{F}$ by \emph{$n$-cutting above $a_k$}. The proof of Proposition~\ref{propncutperiodicfrieze} works quite similar as the one of Proposition~\ref{propnglueperiodicfrieze} by iteratively applying Theorem~\ref{thmcutfrieze}. Note again that, strictly speaking, the operation of $n$-cutting leads to a periodic infinite frieze with period a divisor of $n\!-\!1$, possibly smaller than $n\!-\!1$.

%
\section{Arithmetic friezes}\label{secarithmeticfriezes}
%

It is natural to ask which sequences yield a periodic infinite frieze. A first part of the answer is given in the following section of this article. A complete characterization via triangulations is given by the subsequent article \cite{BPT}, building on our results. We generalize the basic construction for obtaining a finite frieze from triangulated polygons. Instead of triangulations of polygons we consider triangulations of once-punctured discs and associate sequences of positive integers to them thus generating periodic infinite friezes. We provide a geometrical interpretation of the two operations $n$-gluing and $n$-cutting on periodic infinite friezes in terms of triangulations of once-punctured discs. Moreover, we prove a remarkable arithmetic property of these periodic infinite friezes.

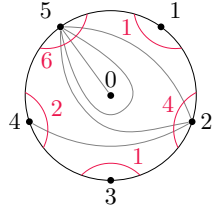
\begin{figure}[t]
\scalebox{0.9}{\begin{tikzpicture}[font=\normalsize] 
\node (a) at (0,0) [fill,circle,inner sep=1pt] {};
\draw (0,0) circle (1.25cm);
\foreach \x in {-144,-72,0,72,144} {
   \begin{scope}[rotate=\x]
    \node (\x) at (0,-1.25) [fill,circle,inner sep=1pt] {};
   \end{scope}
}
\draw (a) node [above] {$0$};
\draw (144) node [above right] {$1$};
\draw (72) node [right] {$2$};
\draw (0) node [below] {$3$};
\draw (-72) node [left] {$4$};
\draw (-144) node [above left] {$5$};
\draw[thin,opacity=0.5] (a) to (-144);
\draw[thin,opacity=0.5,out=-140,in=-80] (0.2,-0.2) to (-144);
\draw[thin,opacity=0.5,out=-30,in=40] (-144) to  (0.2,-0.2);
\draw[thin,opacity=0.5,out=110,in=0] (72) to  (-144);
\draw[thin,opacity=0.5,out=-30,in=-150] (-72) to (72);
\draw[thin,opacity=0.5,out=-170,in=-45] (72) to (-0.3,-0.3) ;
\draw[thin,opacity=0.5,out=135,in=-90] (-0.3,-0.3) to(-144);

\foreach \x in {-144,-72,0,72,144} {
   \begin{scope}[rotate=\x]

\draw[thin,myred,out=45,in=135] (-0.428,-1.175) to (0.428,-1.175);
   \end{scope}
}

\draw (144) node [myred,shift={(-0.5cm,0cm)}] {$1$};
\draw (72) node [myred,shift={(-0.35cm,0.25cm)}] {$4$};
\draw (0) node [myred,shift={(0.4cm,0.35cm)}] {$1$};
\draw (-72) node [myred,shift={(0.4cm,0.2cm)}] {$2$};
\draw (-144) node [myred,shift={(-0.2cm,-0.5cm)}] {$6$};
\end{tikzpicture}}
\caption{A triangulation $\Pi$ of $S_5^1$ with quiddity sequence $q_\Pi=(1,4,1,2,6)$.}\label{figexquidditysequencedisc}
\end{figure}

Let us briefly recall some basic notions of triangulated punctured discs, for more details on triangulations of bordered surfaces with marked points see \cite{{FST}}. For $n\geq 1$, the \emph{once-punctured disc} $S_n^1$ is a closed disc with $n$ marked points on the boundary, numbered by $1,2 ,\cdots n$ in clockwise order, and one marked point in the interior, namely the \emph{puncture} labeled by $0$. An \emph{arc} in $S_n^1$ is a non-boundary, non-self-intersecting curve, connecting two marked points of $S_n^1$. In the sequel $S_n^1$ is always meant to be a once-punctured disc together with a fixed labeling and we only consider isotopy classes of arcs. An arc whose endpoints coincide is called a \emph{loop}. We shall use the following notation for arcs of punctured discs: $0j$ indicates the \emph{bridging arc} connecting the puncture with a marked point $j$ on the boundary. For two marked points $i,j$ on the boundary $ij$ denotes the arc isotopic to the boundary segment going clockwise fromneighbor $i$ to $j$ by ignoring all marked points on the boundary other than $i$ and $j$. Such an arc is called \emph{peripheral}. We use the following convention for marked points: if $i=1$, $i-1$ equals $n$ , and if $i=n$, $i+1$ equals $1$.

Two arcs are said to be \emph{non-crossing} if they have no point of intersection in the interior of $S_n^1$. A maximal collection $\Pi$ of pairwise non-crossing arcs in $S_n^1$ is called a \emph{triangulation} of $S_n^1$. One can easily verify that every triangulation of $S_n^1$ consists of exactly $n$ arcs and cuts $S_n^1$ into $n$ disjoint regions, called \emph{triangles}. Since the number of arcs in $S_n^1$ is finite, there exist only finitely many triangulations of $S_n^1$. When considering symmetries of triangulations, we always assume that the marked points on the boundary of $S_n$ are evenly distributed. For combinatorial reasons it is useful to consider triangulations of punctured discs up to rotation through $\frac{2\pi}{n}$ about the puncture.  If two triangulations of $S_n^1$ are rotation-equivalent they are said to be of the same \emph{shape}. Clearly, there are at most $n$ different triangulations of $S_n^1$ of the same shape, depending on the symmetries the triangulation has. 

Note that every triangulation of $S_n^1$ contains at least one bridging arc. The special case where a triangulation consists entirely of bridging arcs is called \emph{star-triangulation} and denoted by $\Pi_{\ast}$.
  
\begin{defin}
Let  $\Pi$ be a triangulation of $S_n^1$. The \emph{quiddity sequence} $q_{\Pi}$ of $\Pi$ is the finite sequence $(a_1, a_2,\dots,a_n)$ of positive integers, where $a_i$ is the number of connected components of $S_n^1\setminus{\Pi}\cap U$, for $U$ a small neighborhood of $i$. 
\end{defin}

We will see that every $q_{\Pi}$ gives rise to a periodic infinite frieze (Theorem~\ref{thmtriangulationfrieze}). Hence it makes sense to call $q_{\Pi}$ a quiddity sequence. Figure~\ref{figexquidditysequencedisc} gives an example of a triangulation of $S_5^1$ together with its quiddity sequence. 

Clearly, for two triangulations of $S_n^1$ with the same shape the quiddity sequences coincide up to cyclic permutation. Note also that the reflection at a diameter through a marked point  on the boundary of a triangulation $\Pi$ of $S_n^1$ with quiddity sequence \mbox{$q_\Pi=(a_1, a_2,\dots,a_n)$} is a triangulation $\Pi'$ of $S_n^1$  with quiddity sequence $q_{\Pi'}=(a_n, a_{n-1},\dots,a_1)$.

\begin{defin}
Given a triangulation $\Pi$ of $S_n^1$ with quiddity sequence $q_{\Pi}=(a_1, a_2,\dots,a_n)$ a marked point $i\in\{1,2,\dots,n\}$ on the boundary is called \emph{special} with respect to $\Pi$ if $a_i=1$,  i.e.\ if $i$ is incident with exactly one triangle $\tau_i$ in $\Pi$. In this case, $\tau_i$ is called a \emph{special triangle}.
\end{defin}

The following result for triangulations of once-punctured discs and special marked points can easily be checked.

\begin{lem}\label{lemspecialmarkedpoint}
Every triangulation of $S_n^1$ other than the star-triangulation has at least one special marked point. Therefore, its quiddity sequence contains at least one entry $1$. 
\end{lem}

Note that special triangles can be removed from  triangulations of $S_n^1$ to obtain triangulations of $S_{n'}^1$ for $n'<n$. Conversely, we can always add triangles to triangulations. It can easily be verified that this gives two elementary operations on triangulations of once-punctured discs, called \emph{cutting}, and \emph{gluing}, respectively. For later usage we record the effect of these procedures to the quiddity sequences.

\begin{cor}\label{corgluecuttriangle}
Let $\Pi$ be a triangulation of $S_{n}^1$ with quiddity sequence $q_{\Pi}=(a_1,a_2,\dots,a_n)$.

\begin{inparaenum}[$a)$] 
\item\label{corcuttriangle}If $x\in\{1,2,\dots,n\}$ is a special marked point of $\Pi$ and $\Pi_{\setminus x}$ denotes the union of all triangles in $\Pi$ other than the special triangle at $x$, then $\Pi_{\setminus x}$ is a triangulation of $S_{n-1}^1$ with quiddity sequence
$$q_{\Pi_{\setminus x}}=\begin{cases}
(2)&\text{if } n=2,\\
(a_1,\dots,a_{x-2},a_{x-1}\!-\!1,a_{x+1}\!-\!1,a_{x+2},\dots,a_n)&\text{otherwise}.
\end{cases}$$

\item\label{corgluetriangle} If $x$ is a marked point added to the boundary of $S_n^1$ between $i$ and $i+1$ and $\Pi_{\cup x}$ denotes the union of all triangles in $\Pi$ together with the triangle having vertices $i,x$ and $i+1$, then $\Pi_{\cup x}$ is a triangulation of $S_{n+1}^1$ with quiddity sequence
$$q_{\Pi_{\cup x}}=\begin{cases}
(4,1) &\text{if } $n=1$,\\
(a_1,\dots,a_{i-1},a_i\!+\!1,1,a_{i+1}\!+\!1,a_{i+2},\dots,a_n)&\text{otherwise}.
\end{cases}$$
\end{inparaenum}
\end{cor}

\begin{rem}
It is noteworthy that every triangulation with $r$ bridging arcs of $S_n^1$ can be obtained from the star-triangulation on $r$ arcs by gluing triangles successively.
\end{rem}

Observe that using the action of cutting for triangulations of once-punctured discs it is not hard to show inductively that the quiddity sequence provides all the information about the corresponding triangulation.

We now come to one of the main results of this article.

\begin{thm}\label{thmtriangulationfrieze}
Let $\Pi$ be a triangulation of $S_n^1$. Then the quiddity sequence $q_{\Pi}=(a_1, a_2,\dots,a_n)$ of $\Pi$ is a quiddity sequence of an infinite frieze $\mathcal{F}_{\Pi}$ of period $n$.
\end{thm}

\begin{proof}
We prove the result by induction on $n$. For $n=1$, there is only the star-triangulation $\Pi_\ast$ with quiddity sequence $q_{\Pi_\ast}=(2)$ and this is a quiddity sequence $q_{\mathcal{F}_\ast}$ for the basic infinite frieze $\mathcal{F}_\ast$, cf.\ Figure~\ref{figbasicfrieze}, thus $\mathcal{F}_{\Pi_\ast}=\mathcal{F}_\ast$.

Now, for $n\geq 1$, we assume that any triangulation of $S_n^1$ yields an $n$-periodic infinite frieze. Let $\Pi$ be a triangulation of $S_{n+1}^1$ and  \mbox{$q_{\Pi}=(a_1,a_2,\dots,a_{n+1})$} its quiddity sequence. If $\Pi=\Pi_\ast$, the claim follows as in the base case. Otherwise, if $\Pi\ne\Pi_\ast$, there is a special marked point $x$ of $\Pi$ (Lemma~\ref{lemspecialmarkedpoint}). By Corollary~\ref{corgluecuttriangle} $\ref{corcuttriangle})$, $\Pi_{\setminus x}$ is a triangulation of $S_n^1$ with quiddity sequence $q_{\Pi_{\setminus x}}=(a_1,\dots,a_{x-2},a_{x-1}\!-\!1,a_{x+1}\!-\!1,a_{x+2},\dots,a_{n+1})$ (or $(2)$ if $n=1$). By induction, $q_{\Pi_{\setminus x}}$ is the quiddity sequence of an $n$-periodic infinite frieze $\mathcal{F}_{\Pi_{\setminus x}}$. Now we $n$-glue above $(a_{x-1}\!-\!1,a_{x+1}\!-\!1)$ in $\mathcal{F}_{\Pi_{\setminus x}}$ and by \mbox{Proposition~\ref{propnglueperiodicfrieze}}, this gives an infinite frieze $\widehat{\mathcal{F}}_{\Pi_{\setminus x}}^n$ of period $n\!+\!1$ such that $\hat{ q}_{\mathcal{F}_{\Pi_{\setminus x}}}=q_\Pi$. This completes the proof.
\end{proof}

\begin{figure}[t]
\resizebox{.9\linewidth}{!}{\begin{tikzpicture}[font=\normalsize] 
  \matrix(m) [matrix of math nodes,row sep={1.5em,between origins},column sep={1.5em,between origins},nodes in empty cells]{
&&&&&&&&&&&&&&&&&&&&&&&&&&&&&&&&&&&&&\\
&0&&0&&0&&0&&0&&0&&0&&0&&0&&0&&0&&&&&&&&&&&&&&&&\\
&&1&&1&&1&&1&&1&&1&&1&&1&&1&&1&&1&&&&&&&&&&&&&&&\\
&&&1&&4&&1&&2&&6&&1&&4&&1&&2&&6&&1&&&&&&&&&&&&&&\\
&&&&3&&3&&1&&11&&5&&3&&3&&1&&11&&5&&3&&&&&&&&&&&&&\\
&&&&&2&&2&&5&&9&&14&&2&&2&&5&&9&&14&&2&&&&&&&&&&&&\\
&&&\node{\cdots};&&&1&&9&&4&&25&&9&&1&&9&&4&&25&&9&&1&&&\node{\cdots};&&&&&&&&\\
&&&&&&&4&&7&&11&&16&&4&&4&&7&&11&&16&&4&&4&&&&&&&&&&\\
&&&&&&&&3&&19&&7&&7&&15&&3&&19&&7&&7&&15&&3&&&&&&&&&\\
&&&&&&&&&8&&12&&3&&26&&11&&8&&12&&3&&26&&11&&8&&&&&&&&\\
&&&&&&&&&&5&&5&&11&&19&&29&&5&&5&&11&&19&&29&&5&&&&&&&\\
&&&&&&&&&&&2&&18&&8&&50&&18&&2&&18&&8&&50&&18&&2&&&&&&\\
&&&&&&&&&&&&7&&13&&21&&31&&7&&7&&13&&21&&31&&7&&7&&&&&\\
&&&&&&&&&&\node{\cdots};&&&5&&34&&13&&12&&24&&5&&34&&13&&12&&24&&5&&&\node{\cdots};&\\
&&&&&&&&&&&&&&13&&21&&5&&41&&17&&13&&21&&5&&41&&17&&13&&&\\
&&&&&&&&&&&&&&&8&&8&&17&&29&&44&&8&&8&&17&&29&&44&&8&&\\
&&&&&&&&&&&&&&&&3&&27&&12&&75&&27&&3&&27&&12&&75&&27&&3&\\
&&&&&&&&&&&&&&&&&&&&&&&&&&&&&&&&&&&&&\\
&&&&&&&&&&&&&&&&&&&&&&\node[rotate=-6.5,shift={(-0.034cm,-0.08cm)}]  {\ddots};&&&&&&&&&&\node[rotate=-6.5,shift={(-0.034cm,-0.08cm)}]  {\ddots};&&&&&\\
};

\draw (m-18-22) node[myblue,shift={(0cm,0.1cm)}]{$\mathcal{D}$};

\draw[opacity=0,rounded corners, fill=myblue,fill opacity=0.15] ($(m-1-1.south west)+(-0.15cm,0cm)$) -- ($(m-17-17.south west)+(-0.15cm,0cm)$)--($(m-17-27.south west)+(-0.15cm,0cm)$) -- ($(m-1-11.south west)+(-0.15cm,0cm)$) -- cycle;

\draw[semithick,dashed,myblue,opacity=0.25] ($(m-1-1.south west)+(-0.15cm,0cm)$) -- ($(m-17-17.south west)+(-0.15cm,0cm)$);
\draw[semithick,dashed,myblue,opacity=0.25] ($(m-1-11.south west)+(-0.15cm,0cm)$) -- ($(m-17-27.south west)+(-0.15cm,0cm)$);
\draw[semithick,dashed,myblue,opacity=0.25] ($(m-1-21.south west)+(-0.15cm,0cm)$) -- ($(m-17-37.south west)+(-0.1cm,0cm)$);    

\fill[myred,opacity=0.3] (m-4-6) circle (0.25cm);
\fill[myred,opacity=0.3] (m-9-11) circle (0.25cm);
\fill[myred,opacity=0.3] (m-14-16) circle (0.25cm);

\fill[mygreen,opacity=0.3] (m-2-10) circle (0.25cm);
\fill[mygreen,opacity=0.3] (m-7-15) circle (0.25cm);
\fill[mygreen,opacity=0.3]  (m-12-20) circle (0.25cm);
\fill[mygreen,opacity=0.3] (m-17-25) circle (0.25cm);

\end{tikzpicture}}
\caption{The $5$-arithmetic frieze associated to the triangulation of $S_5^1$ shown in Figure~\ref{figexquidditysequencedisc}.}\label{figexfrieze5p2}
\end{figure}
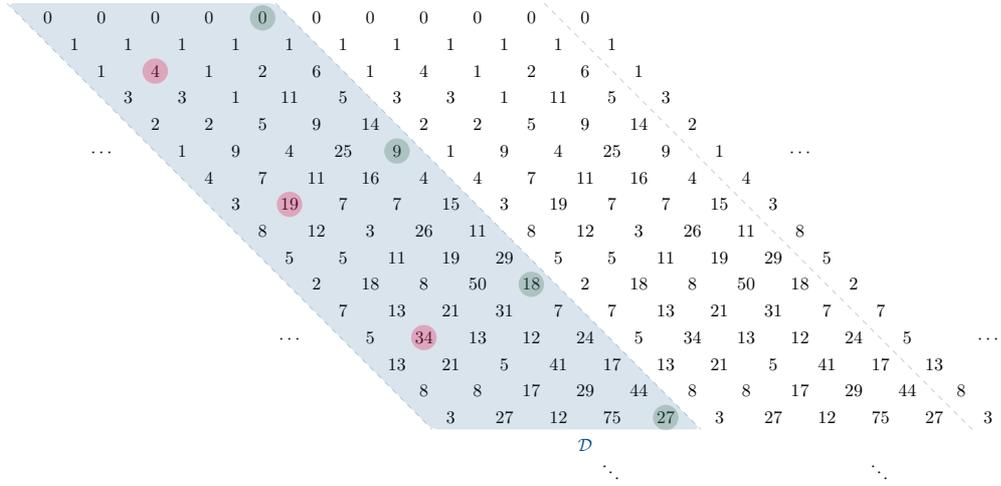

For the triangulation of $S_5^1$ given in Figure~\ref{figexquidditysequencedisc} the associated periodic infinite frieze is illustrated in Figure~\ref{figexfrieze5p2}.

Note that an immediate consequence of Theorem~\ref{thmtriangulationfrieze} and Corollary~\ref{corgluecuttriangle} is that cutting and gluing of triangles for triangulations of once-punctured discs provide a geometric interpretation via triangulations for the operations $n$-cutting and $n$-gluing defined on periodic infinite friezes in Section \ref{seccutglue}.

\begin{cor}\label{corspecialmpfriezeequivalce}
Given a triangulation $\Pi\ne \Pi_\ast$ of $S_{n+1}^1$ with quiddity sequence $q_{\Pi}=(a_1,a_2,\dots,a_{n+1})$ and special marked point $x$, let $\mathcal{F}_{\Pi}$ be the $(n\!+\!1)$-periodic infinite frieze associated to $\Pi$, and let $\mathcal{F}_{\Pi_{\setminus x}}$ be the $n$-periodic infinite frieze associated to $\Pi_{\setminus x}$. Then $\mathcal{F}_{\Pi_{\setminus x}}$ equals $\widecheck{\mathcal{F}}_{\Pi}^{n+1}$ after $(n\!+\!1)$-cutting above $a_x=1$ in $\mathcal{F}_{\Pi}$, and $\mathcal{F}_{\Pi}$ equals $\widehat{\mathcal{F}}_{\Pi_{\setminus x}}^n$ after $n$-gluing above $(a_{x-1}-1,a_{x+1}-1)$ in $\mathcal{F}_{\Pi_{\setminus x}}$.
\end{cor}

As in \cite{BCI} for finite friezes, the next result provides information about the occurrence of the entry $1$ in periodic infinite friezes associated to triangulations of once-punctured discs. In fact, the opposite direction of the following lemma holds also true as we shall see later at the end of this section (Proposition~\ref{propunitentry}).

\begin{lem}\label{lemunitentry}
Let $\mathcal{F}_{\Pi}=(m_{ij})_{j-i\ge -2}$ be an $n$-periodic infinite frieze associated to a triangulation $\Pi$ of $S_n^1$ and let $ij$ be a peripheral arc in $\Pi$. Then for all $l\in\mathbb{Z}$, $m_{(i+1)+ln,(j-1)+ln}=1$ if $j\ge i+2$ and  $m_{(i+1)+ln,(j-1+n)+ln}=1$ otherwise.
\end{lem}

\begin{proof}
It is enough to show the claim for $l=0$ since $\mathcal{F}_{\Pi}$ is $n$-periodic. We use induction on $n$. If $n=1$, there is only the star-triangulation, so let $n=2$.  All triangulations of $S_2^1$ containing a peripheral arc have the same shape thus provide the same $2$-periodic infinite frieze $\mathcal{F}$. W.l.o.g.\ we choose the quiddity sequence such that $a_1=4$ and $a_2=1$. Among one bridging arc we have one peripheral arc, namely $11$ (with $j<i+1$). We need to verify $m_{22}=1.$ By construction, we have $m_{22}=a_2=1$ as desired.

We now assume that the claim holds for any triangulation of $S_n^1$. Let $\Pi\ne\Pi_\ast$ be a triangulation of $S_{n+1}^1$ with quiddity sequence $q_\Pi=(a_1,a_2,\dots,a_{n+1})$ and associated $(n\!+\!1)$-periodic infinite frieze $\mathcal{F}_\Pi=(m_{ij})_{j-i\ge -2}$. By Lemma~\ref{lemspecialmarkedpoint}, $\Pi$ has a special marked point $x\in\{1,2,\dots n+1\}$ and $\Pi_{\setminus x}$ is a triangulation of $S_n^1$ (Corollary~\ref{corgluecuttriangle} $\ref{corcuttriangle}$))  with quiddity sequence $q_{\Pi_{\setminus x}}=(\check a_1,\check a_2,\dots,\check a_n)$ and associated $n$-periodic infinite frieze $\mathcal{F}_{\Pi_{\setminus x}}=(\check m_{ij})_{j-i\ge -2}$, where $\mathcal{F}_{\Pi_{\setminus x}}$ is obtained from $\mathcal{F}_\Pi$ by $(n\!+\!1)$-cutting above $a_x=1$ (Corollary~\ref{corspecialmpfriezeequivalce}). In reverse, we get $\mathcal{F}_\Pi$ from $\mathcal{F}_{\Pi_{\setminus x}}$ by $n$-gluing above $(\check a_{x-1},\check a_{x})$, where $\check a_{x-1}=a_{x-1}-1$ and $\check a_{x}=a_{x+1}-1$.

With the notation of Corollary~\ref{cornglueperiodicfrieze}, if there is a peripheral $ij$ ($i,j\in\{1,2,\dots,n+1\}\setminus \{x\}$) arc in $\Pi$ other than $(x-1)(x+1)$, then $i_xj_x$ is also in $\Pi_{\setminus x}$. So we assume $ij\ne(x-1)(x+1)$ is a peripheral arc in $\Pi$ and we can use the first case of Corollary~\ref{cornglueperiodicfrieze} (for $k=x-1$). Either $j\ge i+2$, in which case $m_{i+1,j-1}=\check m_{(i+1)_{x+1},(j-1)_{x-1}}=\check m_{i_x+1,j_x-1}$ and $j_x\ge i_x+2$, or $j< i+2$, in which case $m_{i+1,j-1+(n+1)}=\check m_{(i+1)_{x+1},(j-1)_{x-1}+n}=\check m_{i_x+1,j_x-1+n}$ and $j_x< i_x+2$. Hence the result follows by induction. Finally, if $i=x-1$ and $j=x+1$, Corollary~\ref{cornglueperiodicfrieze} implies that $m_{xx}=\check m_{x,x-1}$, and by definition, the latter is $1$. This completes the proof.
\end{proof}

\begin{rem}
We already observed that triangulations of $S_n^1$ with the same shape provide the same quiddity sequence, up to cyclic permutation, thus give rise to the same $n$-periodic infinite frieze. In general, the periodic infinite frieze associated to a triangulation of $S_n^1$ has period $n$. But, it might also have shorter periods: if a triangulation of $S_n^1$ has rotational symmetries, the shortest period of the associated infinite frieze is a factor of $n$, as pictured on the right in the figure below. Indeed, we can construct many triangulations of different once-punctured discs giving rise to the same periodic infinite frieze.
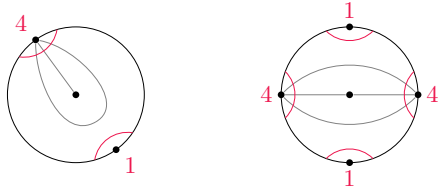
\begin{center}
\scalebox{0.9}{\begin{tikzpicture}[font=\normalsize] 

\node (b) at (-4,0) [fill,circle,inner sep=1pt] {};
\draw (-4,0) circle (1cm);

\foreach \x in {-144,36} {
   \begin{scope}[shift={(-4cm,0cm)},rotate=\x]
    \node (\x) at (0,-1) [fill,circle,inner sep=1pt] {};
   \end{scope}
}

\draw[thin,opacity=0.5] (b) to (-144);
\draw[thin,opacity=0.5,out=210,in=-90] (0.29-4,-0.4) to (-144);
\draw[thin,opacity=0.5,out=-10,in=30] (-144) to (0.29-4,-0.4);

\foreach \x in {-144,36} {
   \begin{scope}[shift={(-4,0)},rotate=\x]

\draw[thin,myred,out=45,in=135] (-0.342,-0.94) to (0.342,-0.94);
   \end{scope}
}
     
\draw (-144) node [myred,above left] {$4$};
\draw (36) node [myred,below right] {$1$};

\node (a) at (0,0) [fill,circle,inner sep=1pt] {};
\draw (0,0) circle (1cm);
\foreach \x in {0,90,180,270} {
   \begin{scope}[rotate=\x]
    \node (\x) at (0,-1) [fill,circle,inner sep=1pt] {};
   \end{scope}
}

\draw[thin,opacity=0.5] (a) to (90);
\draw[thin,opacity=0.5] (a) to (270);
\draw[thin,opacity=0.5,out=45,in=-225] (270) to (90);
\draw[thin,opacity=0.5,out=-45,in=-135] (270) to (90);

\foreach \x in {0,90,180,270} {
   \begin{scope}[rotate=\x]

\draw[thin,myred,out=45,in=135] (-0.342,-0.94) to (0.342,-0.94);
   \end{scope}
}

\draw (0) node [myred,below] {$1$};
\draw (90) node [myred,right] {$4$};
\draw (180) node [myred,above] {$1$};
\draw (270) node [myred,left] {$4$};
\end{tikzpicture}}
\end{center}
\noindent Hence the associated periodic infinite friezes are not uniquely determined. Moreover, let us point out that there are periodic infinite friezes which can not be given by a triangulation of a once-punctured disc. Examples for this fact are the complete infinite frieze with $a>2$ or the periodic infinite frieze in Figure~\ref{figexfrieze5p1} they come from triangulated annuli, see \cite{{BPT}}.
\end{rem}

We now will see that the entries in periodic infinite friezes associated to triangulations of once-punctured discs satisfy a beautiful arithmetic property. For instance, in Figure~\ref{figexfrieze5p2} the numbers marked respectively by red and green circles form a sequence with entries in a {\sc se}-diagonal given always by jumping $5$ entries down. We will show that such sequences have common differences, and thus are increasing arithmetic progression. In Figure~\ref{figexfrieze5p2}, the indicated sequences have common differences $15$ and $9$, respectively.

\begin{defin}
For an infinite frieze $\mathcal{F}=(m_{ij})_{j-i\ge -2}$ and a positive integer $r\ge 1$, let \linebreak$d_{ik}:=m_{i,(i+k-3)+r}-m_{i,i+k-3}$ for all $i\in\mathbb{Z}$ and $k\in\{1,2,\dots,r\}$. We say that $\mathcal{F}$ is \emph{$r$-arithmetic} if $m_{i,(i+k-3)+(l+1)r}-m_{i,(i+k-3)+lr}=d_{ik}$ is satisfied for all $l \geq 0$ and every $d_{ik}$, where the $d_{ik}$ are the \emph{common differences} for $\mathcal{F}$.
\end{defin}

\begin{prop}\label{proptriangulationfriezearithmetic}
 Every $n$-periodic infinite frieze $\mathcal{F}_\Pi$ associated to a triangulation $\Pi$ of $S_n^1$ is $n$-arithmetic.
\end{prop}

\begin{proof}
Clearly, if $\Pi=\Pi_\ast$ is the star-triangulation of $S_n^1$, the claim is true with $d_{ik}=n$ for all $i\in\mathbb{Z}$ and $k\in\{1,2,\dots, n\}$. In particular the claim is true for $n=1$.

We proceed with the inductive step and assume the claim holds for every $n$-periodic infinite frieze associated to a triangulation of $S_n^1$. Now we consider a triangulation $\Pi\ne\Pi_\ast $ of $S_{n+1}^1$ with quiddity sequence $q_\Pi=(a_1,a_1,\dots a_{n+1})$ and associated $(n\!+\!1)$-periodic infinite frieze $\mathcal{F}_\Pi=(m_{ij})_{j-i\ge -2}$. Lemma~\ref{lemspecialmarkedpoint} implies that $\Pi$ contains a special marked point $x\in\{1,2,\dots, n+1\}$ such that $a_x=1$, and $\Pi_{\setminus x}$ is a triangulation of $S_n^1$ with quiddity sequence $q_{\Pi_{\setminus x}}=(\check a_1,\check a_2,\dots,\check a_n)$ as in Corollary~\ref{corgluecuttriangle}. By Corollary~\ref{corspecialmpfriezeequivalce}, $\mathcal{F}_{\Pi }$ is obtained from $\mathcal{F}_{\Pi_{\setminus x}}=(\check m_{ij})_{j-i\ge -2}$ by $n$-gluing above $(\check a_{x-1},\check a_x)$.

Clearly, it is enough to show the claim for a fundamental domain of $\mathcal{F}_\Pi$. So we choose \mbox{$i,k\in\{1,2,\dots,n+1\}$} and show that $\hat m_{i,(i+k-3)+(l+1)(n+1)}- \hat m_{i,(i+k-3)+l(n+1)}$ equals a common difference for all $l\ge 0$. By using Corollary~\ref{cornglueperiodicfrieze} (for $k=x-1),$ we are abel to express the entries in $\mathcal{F}_{\Pi }$ in terms of entries in $\mathcal{F}_{\Pi_{\setminus x}}$. Doing this, we have to distinguish four cases. If $i\not \equiv x+1$ and $k\not \equiv x-i+2$ (modulo $n+1$), it follows that $m_{i,(i+k-3)+(l+1)(n+1)}-m_{i,(i+k-3)+l(n+1)}=\check m_{i_{x+1},(i+k-3)_{x-1}+(l+1)n}-\check m_{i_{x+1},(i+k-3)_{x-1}+ln}$, and by induction, the latter equals a common differences for $\mathcal{F}_{\Pi_{\setminus x}}$ for all $l\geq 0$ which gives $d_{ik}$. Suppose $i\not \equiv x+1$ and $k\equiv x-i+2$.  Then $m_{i,(i+k-3)+(l+1)(n+1)}- m_{i,(i+k-3)+l(n+1)}=\check m_{i_{x+1},(i+k-3)_{x-1}-1+(l+1)n}+\check m_{i_{x+1},(i+k-3)_{x-1}+(l+1)n}-\check m_{i_{x+1},(i+k-3)_{x-1}-1+ln}-\check m_{i_{x+1},(i+k-3)_{x-1}+ln}$, that is equal to the sum of two fixed common differences for $\mathcal{F}_{\Pi_{\setminus x}}$ for all $l\geq 0$ (inductive hypothesis), which gives $d_{ik}$. Similarly, for $\equiv  x+1$ and $k\not \equiv x-i+2$, $d_{ik}$ is also the sum of two fixed common differences  for $\mathcal{F}_{\Pi_{\setminus x}}$, and if $i\equiv  x+1$ and $k \equiv x-i+2$, $d_ik$ is determined by four fixed common differences  for $\mathcal{F}_{\Pi_{\setminus x}}$. Hence $\mathcal{F}_{\Pi }$ satisfies the arithmetic property for $r=n+1$.
\end{proof}

An immediate corollary of this proposition is the following.

\begin{cor}
Let $\mathcal{F}$ be an $n$-periodic infinite frieze. Let $\widehat{\mathcal{F}}^n$ be some $(n\!+\!1)$-periodic infinite frieze obtained from $\mathcal{F}$ by $n$-gluing and $\widecheck{\mathcal{F}}^n$ be some $(n\!-\!1)$-periodic infinite frieze obtained from $\mathcal{F}$ by $n$-cutting if defined. If $\mathcal{F}$ is $n$-arithmetic, then $\widehat{\mathcal{F}}^n$ is $(n\!+\!1)$-arithmetic and $\widecheck{\mathcal{F}}^n$ is $(n\!-\!1)$-arithmetic.
\end{cor}

Clearly, finite friezes are not arithmetic in our setup. So far there are no known examples of non-periodic arithmetic infinite friezes. Thus, from now on we shall assume an \emph{$n$-arithmetic frieze} to be infinite and $n$-periodic.

Note that an $n$-periodic infinite frieze being $n$-arithmetic this means that every {\sc se}-diagonal can be split into $n$ increasing arithmetic progressions. Moreover, since a fundamental domain is given by $n$ {\sc se}-diagonals we have $n^2$ increasing arithmetic progressions overall that occur. Thus an entry $1$ can only appear within the first $n\!+\!1$ rows of an $n$-periodic infinite frieze. Using this fact and very similar ideas as in the proof of Lemma~\ref{lemunitentry}, it is easy to prove the opposite direction of Lemma~\ref{lemunitentry}. We leave the details to the reader. We get the following extended version. 

\begin{prop}\label{propunitentry}
Let $\mathcal{F}_{\Pi}=(m_{ij})_{j-i\ge -2}$ be an $n$-periodic infinite frieze associated to a triangulation $\Pi$ of $S_n^1$. Then $ij$ is a peripheral arc in $\Pi$ if and only if $m_{(i+1)+ln,(j-1)+ln}=1$ whenever $j\ge i+2$, or $m_{(i+1)+ln,(j-1+n)+ln}=1$ otherwise.
\end{prop}

%
\section{Description via matching numbers}\label{secmachingnumbers}
%

In this section we shall focus on arithmetic friezes associated to triangulations of once-punctured discs. We present a combinatorial interpretation of the numbers in such an arithmetic friezes using matching numbers between vertices and triangles. In order to do this, we introduce periodic triangulations of strips which we can interpret as triangulations of once-punctured discs. We will show that the number of matchings between vertices of a strip and triangles in a periodic triangulation of it are exactly the entries of the associated arithmetic frieze. Thus  we receive an analogous result as the one by Broline, Crowe and Isaacs in \cite{BCI} for finite friezes and triangulated polygons.

\subsection{Periodic triangulations of strips}

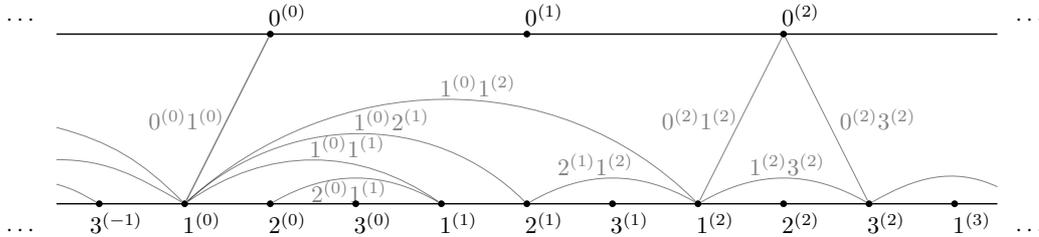
\begin{figure}[b]
\scalebox{0.9}{\begin{tikzpicture}[scale=1.25,font=\normalsize]

\draw (9.5,1) node [above,shift={(0.5cm,0cm)}] {$\cdots$};
\draw (-1.5,1) node [above,shift={(-0.5 cm,0cm)}] {$\cdots$};
\draw[semithick] (-1.5,1) to (9.5,1);

\draw (9.5,-1) node [below,shift={(0.5cm,-0.2cm)}] {$\cdots$};
\draw (-1.5,-1) node [below,shift={(-0.5 cm,-0.2 cm)}] {$\cdots$};
\draw[semithick] (-1.5,-1) to (9.5,-1);

\node (a) at (1,1) [fill,circle,inner sep=1pt] {};
\node (b) at (4,1) [fill,circle,inner sep=1pt] {};
\node (c) at (7,1) [fill,circle,inner sep=1pt] {};

\draw (a) node [above,shift={(0.25cm,0cm)}] {$0^{(0)}$};
\draw (b) node [above,shift={(0.25cm,0cm)}] {$0^{(1)}$};
\draw (c) node [above,shift={(0.25cm,0cm)}] {$0^{(2)}$};

\foreach \x in {-1,0,1,2,3,4,5,6,7,8,9} {
   \begin{scope}[shift={(\x cm, 0 cm)}]
    \node (\x) at (0,-1) [fill,circle,inner sep=1pt] {};
   \end{scope}
}

\draw (-1) node [below,shift={(0.25cm,0cm)}] {$3^{(-1)}$};
\draw (0) node [below,shift={(0.25cm,0cm)}] {$1^{(0)}$};
\draw (1) node [below,shift={(0.25cm,0cm)}] {$2^{(0)}$};
\draw (2) node [below,shift={(0.25cm,0cm)}] {$3^{(0)}$};
\draw (3) node [below,shift={(0.25cm,0cm)}] {$1^{(1)}$};
\draw (4) node [below,shift={(0.25cm,0cm)}] {$2^{(1)}$};
\draw (5) node [below,shift={(0.25cm,0cm)}] {$3^{(1)}$};
\draw (6) node [below,shift={(0.25cm,0cm)}] {$1^{(2)}$};
\draw (7) node [below,shift={(0.25cm,0cm)}] {$2^{(2)}$};
\draw (8) node [below,shift={(0.25cm,0cm)}] {$3^{(2)}$};
\draw (9) node [below,shift={(0.25cm,0cm)}] {$1^{(3)}$};

\draw[thin,opacity=0.5,out=-15,in=135] (-1.5,-0.1) to (0);
\draw[thin,opacity=0.5,out=0,in=150] (-1.5,-0.48) to (0);
\draw[thin,opacity=0.5,out=-15,in=150] (-1.5,-0.77) to (-1);

\draw[semithick,opacity=0.5] (0) to node [shift={(-0.6cm,0 cm)}] {$0^{(0)}1^{(0)}$} (a);
\draw[thin,opacity=0.5,out=42.5,in=135] (0) to node [sloped,shift={(0.5cm,0.2 cm)}] {$1^{(0)}2^{(1)}$} (4);
\draw[thin,opacity=0.5,out=30,in=150] (1) to  node [sloped,shift={(-0.1cm,-0.2 cm)}] {$2^{(0)}1^{(1)}$}  (3);
\draw[thin,opacity=0.5,out=35,in=145] (0) to node [sloped,shift={(0.5cm,0.2cm)}] {$1^{(0)}1^{(1)}$} (3);
\draw[thin,opacity=0.5,out=30,in=150] (4) to node [sloped,shift={(-0.25cm,0.2 cm)}] {$2^{(1)}1^{(2)}$}  (6);
\draw[thin,opacity=0.5] (6) to node [shift={(-0.6cm,0 cm)}] {$0^{(2)}1^{(2)}$} (c);
\draw[thin,opacity=0.5,out=30,in=150] (6) to node [sloped,shift={(0.05cm,0.2 cm)}] {$1^{(2)}3^{(2)}$}  (8);
\draw[thin,opacity=0.5] (8) to node [shift={(0.75cm,0 cm)}] {$0^{(2)}3^{(2)}$} (c);

\draw[thin,opacity=0.5,out=30,in=155] (8) to (9.5,-0.8) ;

\draw[thin,opacity=0.5,out=42.5,in=135] (0) to node [sloped,shift={(0.5cm,0.2 cm)}] {$1^{(0)}1^{(2)}$} (6);

\end{tikzpicture}}

\caption{A triangulation of the strip $\mathcal{U}_3$.}\label{figextriangulationstrip}
 \end{figure}
 
\begin{defin}
For $n\geq 1$, the \emph{strip} $\mathcal{U}_n$ in $\mathbb{R}^2$ is the Cartesian product of the real numbers and a closed interval with two disjoint countably infinite set of vertices on the upper, and on the lower boundary, respectively. The vertices on the upper boundary are labeled by $\{0^{(k)}\mid k\in\mathbb{Z}\}$ and the vertices on the lower boundary are labeled in groups of $n$ vertices by $\{1^{(k)},\dots,n^{(k)}\mid  k\in\mathbb{Z}\}$. The vertices are arranged such that $0^{(k)}$ lies above the vertices $1^{(k)},\dots,n^{(k)}$ and $k$ increases to the right, see Figure~\ref{figextriangulationstrip}.
\end{defin}

Note that the vertices on the lower boundary correspond to $\mathbb{Z}$ successive copies of the $n$ marked points on the boundary of the once-punctured disc $S_n^1$, whereas the vertices on the upper boundary correspond to $\mathbb{Z}$ copies of the puncture. Since we are mainly interested in combinatorics we assume that the vertices are evenly distributed. Throughout this article we use the following convention for vertices on the lower boundary of $\mathcal{U}_n$, namely $(i+ln)^{(k)}=(i)^{(k+l)}$ for $1\leq i\leq n$, $l\in\mathbb{Z}$. Moreover, an order on the vertices on the lower boundary of $\mathcal{U}_n$ is defined by $i^{(k)}\leq j^{(l)}$ if and only if either $k<l$, or $k=l$ and $i\leq j$, i.e.\ if and only if $i^{(k)}$ is to the left of $j^{(l)}$.

\begin{defin}\label{defsarc}
An \emph{arc} in $\mathcal{U}_n$ is a non-self-intersecting curve, up to isotopy, connecting two vertices of $\mathcal{U}_n$ such that 

\begin{inparaenum}[(\text{A}1)] 
\item at least one vertex belongs to the lower boundary of $\mathcal{U}_n$,

\item the two vertices are neither equal nor neighbors,

\item if one vertex belongs to the upper boundary of $\mathcal{U}_n$, then the superscripts of the two vertices are equal.
\end{inparaenum}

An arc in $\mathcal{U}_n$ connecting two vertices on the lower boundary is called \emph{peripheral}, it is called \emph{bridging} otherwise. 
\end{defin}

\begin{rem}
Compared with the model of the strip used in \cite{{HJ}} in this article we exclude arcs connecting two vertices on the upper boundary. Moreover, we also assign a unique bridging arc to every vertex on the lower boundary, see (A3) in Definition~\ref{defsarc}. These additional conditions are motivated by the idea to identify the vertices on the upper boundary with a single vertex, where all boundary segments in between two vertices vanish. 
\end{rem}

We use the following notation for arcs of strips: the peripheral arc with vertex $i^{(k)}$ on the left and vertex $j^{(l)}$ on the right is denoted by $i^{(k)}j^{(l)}$, in this case either $k<l$, or $i\leq j-2$ for $k=l$. Moreover, $0^{(k)}j^{(k)}$ denotes the unique bridging arc connecting the vertex $j^{(k)}$ on the lower boundary with the vertex $0^{(k)}$ on the upper boundary.

\begin{figure}[t]
\resizebox{0.9\linewidth}{!}{\begin{tikzpicture}[font=\normalsize]

\draw (12.5,1.25) node [above,shift={(0.5cm,0cm)}] {$\cdots$};
\draw (-1.5,1.25) node [above,shift={(-0.5 cm,0cm)}] {$\cdots$};
\draw[semithick] (-1.5,1.25) to (12.5,1.25);

\node (a) at (-1,1.25) [fill,myblue,circle,inner sep=1pt] {};
\node (b) at (4,1.25) [fill,myblue,circle,inner sep=1pt] {};
\node (c) at (9,1.25) [fill,circle,inner sep=1pt] {};

\draw (a) node [myblue,above,shift={(0.25cm,0cm)}] {$0^{(-1)}$};
\draw (b) node [myblue,above,shift={(0.25cm,0cm)}] {$0^{(0)}$};
\draw (c) node [above,shift={(0.25cm,0cm)}] {$0^{(1)}$};

\foreach \x in {-1,0,1,2,3,4,5,6,7,8,9,10,11,12} {
   \begin{scope}[shift={(\x cm, 0 cm)}]
    \node (\x) at (0,-1) [fill,circle,inner sep=1pt] {};
   \end{scope}
}

\draw (12.5,-1) node [below,shift={(0.5cm,-0.2cm)}] {$\cdots$};
\draw (-1.5,-1) node [below,shift={(-0.5 cm,-0.2 cm)}] {$\cdots$};
\draw[semithick] (-1.5,-1) to (12.5,-1);

\draw (-1) node [myblue,below,shift={(0.25cm,0cm)}] {$4^{(-1)}$};
\draw (0) node [myblue,below,shift={(0.25cm,0cm)}] {$5^{(-1)}$};
\draw (1) node [myblue,below,shift={(0.25cm,0cm)}] {$1^{(0)}$};
\draw (2) node [myblue,below,shift={(0.25cm,0cm)}] {$2^{(0)}$};
\draw (3) node [myblue,below,shift={(0.25cm,0cm)}] {$3^{(0)}$};
\draw (4) node [myblue,below,shift={(0.25cm,0cm)}] {$4^{(0)}$};
\draw (5) node [myblue,below,shift={(0.25cm,0cm)}] {$5^{(0)}$};
\draw (6) node [below,shift={(0.25cm,0cm)}] {$1^{(1)}$};
\draw (7) node [below,shift={(0.25cm,0cm)}] {$2^{(1)}$};
\draw (8) node [below,shift={(0.25cm,0cm)}] {$3^{(1)}$};
\draw (9) node [below,shift={(0.25cm,0cm)}] {$4^{(1)}$};
\draw (10) node [below,shift={(0.25cm,0cm)}] {$5^{(1)}$};
\draw (11) node [below,shift={(0.25cm,0cm)}] {$1^{(2)}$};
\draw (12) node [below,shift={(0.25cm,0cm)}] {$2^{(2)}$};

\draw[opacity=0, fill=myblue,fill opacity=0.15] (0,-1) -- (-1,1.25) -- (4,1.25) -- (5,-1) -- cycle;
\foreach \x in {0,1,2,3,4,5} {
   \begin{scope}[shift={(\x cm, 0 cm)}]
    \node (\x) at (0,-1) [myblue,fill,circle,inner sep=1pt] {};
   \end{scope}
}
\node (a) at (-1,1.25) [fill,myblue,circle,inner sep=1pt] {};
\node (b) at (4,1.25) [fill,myblue,circle,inner sep=1pt] {};

\draw (a) node [myblue,shift={(2.5cm,-0.5cm)}] {$\mathcal{P}$};


\draw[myblue,thin] (0) to node [shift={(0.8cm,0.15 cm)}] {$0^{(-1)}5^{(-1)}$} (a);
\draw[myblue,thin,out=42.5,in=135] (0) to node [sloped,shift={(0.3cm,0.2 cm)}] {$5^{(-1)}5^{(0)}$} (5);
\draw[myblue,thin,out=30,in=150] (0) to node [sloped,shift={(0.5cm,0.2 cm)}] {$5^{(-1)}2^{(0)}$}  (2);
\draw[myblue,thin,out=35,in=145] (2) to node [sloped,shift={(-0.45cm,0.15 cm)}] {$2^{(0)}5^{(0)}$} (5);
\draw[myblue,thin,out=30,in=150] (2) to node [sloped,shift={(0cm,-0.15 cm)}] {$2^{(0)}4^{(0)}$}  (4);
\draw[myblue,thin] (5) to node [shift={(0.6cm,0.15 cm)}] {$0^{(0)}5^{(0)}$} (b);

\draw[thin,opacity=0.5,out=42.5,in=135] (5) to (10);
\draw[thin,out=30,in=150,myred,dash pattern=on 1pt off 4pt on 6pt off 4pt] (5) to node [sloped,shift={(0.5cm,0.2 cm)}] {$5^{(0)}2^{(1)}$} (7);
\draw[thin,opacity=0.5,out=35,in=145] (7) to (10);
\draw[thin,opacity=0.5,out=30,in=150] (7) to (9);

\draw[thin,opacity=0.5] (10) to (c);
\draw[thin,out=30,in=150,myred,dash pattern=on 1pt off 4pt on 6pt off 4pt] (10) to node [sloped,shift={(0.5cm,0.2 cm)}] {$5^{(1)}2^{(2)}$} (12);
\draw[thin,opacity=0.5,out=42.5,in=180] (10) to (12.5,0);
\draw[thin,opacity=0.5,out=35,in=200] (12) to (12.5,-0.72);
\draw[thin,opacity=0.5,out=30,in=195] (12) to (12.5,-0.77);

\draw[thin,opacity=0.5,out=-15,in=135] (-1.5,-0.1) to (0);
\draw[thin,opacity=0.5,out=0,in=150] (-1.5,-0.48) to (0);
\draw[thin,opacity=0.5,out=-15,in=150] (-1.5,-0.77) to (-1);

\end{tikzpicture}}
\caption{A $5$-periodic triangulation of $\mathcal{U}_5$ with fundamental domain $\mathcal{P}$ asso\-ciated to the triangulation of $S_5^1$ given in Figure~\ref{figexquidditysequencedisc}.}\label{figextriangulationstrip5p1}
\end{figure}
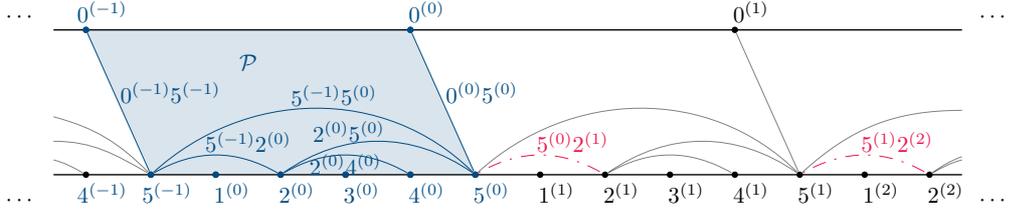

\begin{defin}
Two arcs in $\mathcal{U}_n$ are called \emph{non-crossing} if they have no point of intersection in the interior of $\mathcal{U}_n$. A \emph{triangulation} of $\mathcal{U}_n$ is a maximal collection $\mathcal{T}$ of pairwise non-crossing arcs in $\mathcal{U}_n$.
\end{defin}

Note that $\mathcal{U}_n$ is partitioned by a triangulation into regions called \emph{triangles}. These have three, or more sides. If we identify all vertices on the upper boundary, the segments on this boundary vanishes and all domains are $3$-sided. Figure~\ref{figextriangulationstrip} shows part of a triangulation of $\mathcal{U}_3$.

\begin{defin}
A triangulation $\mathcal{T}=\mathcal{T}_n$ of $\mathcal{U}_n$ is called $n$-\emph{periodic} if  there exists an \mbox{$(n\!+\!3)$-gon} $\mathcal{P}$ in $\mathcal{T}_n$ such that $\mathcal{T}_n$ is covered by iteratively performing an appropriate translation of $\mathcal{P}$ in both horizontal directions. We say that $\mathcal{P}$ is a \emph{fundamental domain} for $\mathcal{T}_n$.
\end{defin}

Note that we actually ignore the fact that covering the whole triangulation by \mbox{$(n\!+\!3)$-gons} leads to very small overlaps, as with every translation we place a bridging arc on a bridging arc.

While a fundamental domain for a given periodic triangulation of a strip is not unique, they all give rise to the entire triangulation. For fixed $n$, the number of different fundamental domains equals, up to translation, the number of bridging arcs ending at a vertex on the upper boundary. Figure~$\ref{figextriangulationstrip5p1}$ shows an example of a $5$-periodic triangulation of $\mathcal{U}_5$ with a fundamental domain given by an octagon, where in this example the fundamental domain is unique up to translation.\\

We call the periodic triangulation consisting only of bridging arcs \emph{star-triangulation}, denoted by $\mathcal{T}_\ast$. See Figure~\ref{figbstartriangulationstrip} where $\mathcal{P}_0$ is one of the $n$ possible choices for a fundamental domain.

\begin{rem}\label{remfundamentaldomainstrip}
In general, if a triangulation of a strip has translational symmetry, bridging arcs occur repetitively in it. Since we restrict to fundamental domains for an $n$-periodic triangulation $\mathcal{T}_n$ of $\mathcal{U}_n$ given by $(n\!+\!3)$-gons it follows by definition that a fundamental domain $\mathcal{P}$ for $\mathcal{T}_n$ looks as in the following figure
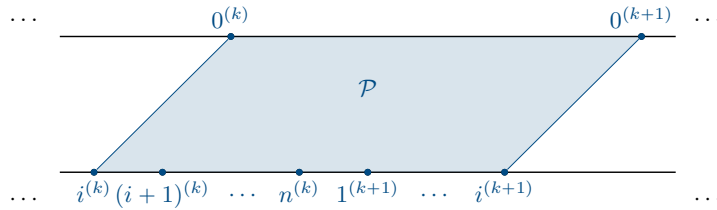
\begin{center}
\scalebox{0.9}{\begin{tikzpicture}[scale=1,font=\normalsize]

\node (a) at (2,1) [fill,circle,inner sep=1pt] {};
\node (b) at (8,1) [fill,circle,inner sep=1pt] {};

\draw (a) node [myblue,above] {$0^{(k)}$};
\draw (b) node [myblue,above] {$0^{(k+1)}$};
\draw (8.5,1) node [above,shift={(0.5 cm,0.1 cm)}] {$\dots$};
\draw (-0.5,1) node [above,shift={(-0.5 cm,0.1 cm)}] {$\dots$};
\draw[semithick] (-0.5,1) to (8.5,1);

\foreach \x in {0,1,3,4,6} {
   \begin{scope}[shift={(\x cm, 0 cm)}]
    \node (\x) at (0,-1) [fill,circle,inner sep=1pt] {};
   \end{scope}
}
\draw (0) node [myblue,below] {$i^{(k)}$};
\draw (1) node [myblue,below] {($i+1)^{(k)}$};
\node at (2.2cm,-1.15cm) [myblue,below] {$\cdots$};
\draw (3) node [myblue,below] {$n^{(k)}$};
\draw (4) node [myblue,below] {$1^{(k+1)}$};
\node at (5cm,-1.15cm) [myblue,below] {$\cdots$};
\draw (6) node [myblue,below] {$i^{(k+1)}$};

\draw (8.5,-1) node [below,shift={(0.5 cm,-0.2 cm)}] {$\cdots$};
\draw (-0.5,-1) node [below,shift={(-0.50 cm,-0.2 cm)}] {$\cdots$};
\draw[semithick] (-0.5,-1) to (8.5,-1);

\draw[opacity=0,fill=myblue,fill opacity=0.15] (0,-1) -- (2,1) -- (8,1) -- (6,-1) -- cycle;
\node (a) at (2,1) [myblue,fill,circle,inner sep=1pt] {};
\node (b) at (8,1) [myblue,fill,circle,inner sep=1pt] {};
\foreach \x in {0,1,3,4,6} {
   \begin{scope}[shift={(\x cm, 0 cm)}]
    \node (\x) at (0,-1) [myblue,fill,circle,inner sep=1pt] {};
   \end{scope}
}
\draw (a) node [myblue,shift={(2cm,-0.75cm)}] {$\mathcal{P}$};

\draw[myblue,thin] (0) to (a);
\draw[myblue,thin] (6) to (b);

\end{tikzpicture}}
\end{center}

\noindent for some $i\in\{1,2\dots,n\}$ and some integer $k$. In particular, for every integer $k$ there is a bridging arc at $0^{(k)}$. Moreover, the interior of $\mathcal{P}$ contains $n\!-\!1$ pairwise non-crossing arcs dividing $\mathcal{P}$ into $n$ triangles. One of these has four, and all the others have three distinct sides. Recall, our philosophy of viewing the vertices on the upper boundary to be a single vertex gives a suitable understanding of the quadrilateral in $\mathcal{P}$ as a triangle and explains why we ignore  segments of the upper boundary.
\end{rem}

\begin{figure}[t]
\scalebox{0.9}{\begin{tikzpicture}[scale=1,font=\normalsize] 

\node (a) at (2,1) [fill,circle,inner sep=1pt] {};
\node (b) at (7,1) [fill,circle,inner sep=1pt] {};

\draw (a) node [myblue,above] {$0^{(0)}$};
\draw (b) node [myblue,above] {$0^{(1)}$};
\draw (10.5,1) node [above,shift={(0.5 cm,0.1 cm)}] {$\dots$};
\draw (-0.5,1) node [above,shift={(-0.5 cm,0.1 cm)}] {$\dots$};
\draw[semithick] (-0.5,1) to (10.5,1);

\foreach \x in {0,1,3,4,5,6,8,9,10} {
   \begin{scope}[shift={(\x cm, 0 cm)}]
    \node (\x) at (0,-1) [fill,circle,inner sep=1pt] {};
   \end{scope}
}
\draw (0) node [myblue,below] {$1^{(0)}$};
\draw (1) node [myblue,below] {$2^{(0)}$};
\node at (1.8cm,-1.15cm) [myblue,below] {$\cdots$};
\draw (3) node [myblue,below,shift={(-0.05 cm,0 cm)}] {$(n-1)^{(0)}$};
\draw (4) node [myblue,below] {$n^{(0)}$};
\draw (5) node [myblue,below] {$1^{(1)}$};
\draw (6) node [below] {$2^{(1)}$};
\node at (6.8cm,-1.15cm) [below] {$\cdots$};
\draw (8) node [below,shift={(-0.05 cm,0 cm)}] {$(n-1)^{(1)}$};
\draw (9) node [below] {$n^{(1)}$};
\draw (10) node [below] {$1^{(2)}$};
\draw (10.5,-1) node [below,shift={(0.5 cm,-0.2 cm)}] {$\cdots$};
\draw (-0.5,-1) node [below,shift={(-0.50 cm,-0.2 cm)}] {$\cdots$};
\draw[semithick] (-0.5,-1) to (10.5,-1);

\draw[opacity=0,fill=myblue,fill opacity=0.15] (0,-1) -- (2,1) -- (7,1) -- (5,-1) -- cycle;
\node (a) at (2,1) [myblue,fill,circle,inner sep=1pt] {};
\node (b) at (7,1) [myblue,fill,circle,inner sep=1pt] {};
\foreach \x in {0,1,3,4,5} {
   \begin{scope}[shift={(\x cm, 0 cm)}]
    \node (\x) at (0,-1) [myblue,fill,circle,inner sep=1pt] {};
   \end{scope}
}

\draw (a) node [myblue,shift={(2.5cm,-0.75cm)}] {$\mathcal{P}_0$};

\draw[myblue,thin] (0) to (a);
\draw[myblue,thin] (1) to (a);
\draw[myblue,thin] (3) to (a);
\draw[myblue,thin] (4) to (a);

\draw[myblue,thin] (5) to (b);
\draw[thin,opacity=0.5] (6) to (b);
\draw[thin,opacity=0.5] (8) to (b);
\draw[thin,opacity=0.5] (9) to (b);

\draw[thin,opacity=0.5] (10) to (10.5,-0.5);
\end{tikzpicture}}
\caption{The star-triangulation $\mathcal{T}_\ast$ of $\mathcal{U}_n$ with a fundamental domain $\mathcal{P}_0$.}\label{figbstartriangulationstrip}
\end{figure}
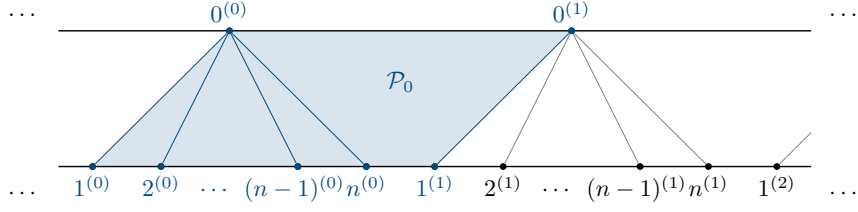

The next lemma points out which arcs can appear in a periodic triangulation of a strip. It is an immediate consequence of the definitions.

\begin{lem}\label{lemarcinfundamentaldomain}
Let $i^{(k)}j^{(l)}$ be an arc in a $n$-periodic triangulation $\mathcal{T}_n$ of $\mathcal{U}_n$. Then $l$ is either $k$, or $k+1$. Moreover, $i^{(k)}j^{(k)}\in \mathcal{T}_n$ implies $i<j$, and $i^{(k)}j^{(k+1)}\in \mathcal{T}_n$ implies $i\geq j$. 
\end{lem}

\begin{defin}
Let $\mathcal{P}$ be a fundamental domain for an $n$-periodic triangulation $\mathcal{T}_n$ of $\mathcal{U}_n$ and let  $\gamma$ be an arc in $\mathcal{P}$. An arc $\bar\gamma$ in $\mathcal{T}_n$ is called \emph{$n$-translate} of $\gamma$ if $\bar\gamma$ corresponds to $\gamma$ in a translated copy of $\mathcal{P}$.
\end{defin}

In \mbox{Figure~\ref{figextriangulationstrip5p1}} two $5$-translates of the arc $1^{(0)}3^{(0)}$ are indicated red. In particular, for an arc $\gamma=i^{(k)}j^{(l)}$ in an $n$-periodic triangulation of $\mathcal{U}_n$, Lemma~\ref{lemarcinfundamentaldomain} tells us that the family of its \mbox{$n$-translates} is given by
$$\bigl\{i^{(k+m)}j^{(l+m)}\bigr\}_{m\in\mathbb{Z}\setminus\{0\}}=
\begin{cases}
\bigl\{i^{(k')}j^{(k')}\bigr\}_{k'\in\mathbb{Z}\setminus\{k\}}&\text{if } i<j,\\
\bigl\{i^{(k')}j^{(k'+1)}\bigr\}_{k'\in\mathbb{Z}\setminus\{k\}} &\text{if } i\geq j.
\end{cases}$$

Note that given an $n$-periodic triangulation $\mathcal{T}_n$ of $\mathcal{U}_n$, the number of triangles in $\mathcal{T}_n$ incident with a vertex $i^{(k)}$ of $\mathcal{U}_n$ does not depend on $k$. This fact is to keep in mind throughout this section. In particular, the next definition already makes use of it.

\begin{defin}
The \emph{quiddity sequence} $q_{\mathcal{T}_n}$ of an $n$-periodic triangulation $\mathcal{T}_n$ of $\mathcal{U}_n$ is a finite sequence $(a_1,a_2,\dots, a_n)$ of positive integers, where $a_i$ is the number of triangles incident with the vertex $i^{(k)}$, on the lower boundary of $\mathcal{U}_n$. A vertex $i^{(k)}$ on the lower boundary is called \emph{special} with respect to $\mathcal{T}_n$ if $a_i=1$.
\end{defin}

Note that for a fixed $n$, the quiddity sequence of a periodic triangulation of a strip is only determined up to cyclic equivalence.

\FloatBarrier
\subsection{From triangulations of once-punctured discs to periodic triangulations of strips}\label{secPsi}

Our next goal is to give a bijection between triangulations of once-punctured discs and periodic triangulations of strips. First, from a triangulation $\Pi$ of $S_n^1$ we construct an $n$-periodic triangulation $\mathcal{T}_n=\Phi(\Pi)$ of $\mathcal{U}_n$ by associating a family of arcs in $\mathcal{U}_n$ to every arc of $\Pi$.

\begin{figure}[t]
\scalebox{0.9}{\begin{tikzpicture}[font=\normalsize] 
\node (a) at (5.5,3.5) [fill,circle,inner sep=1pt] {};
\draw (5.5,3.5) circle (1cm);
\foreach \x in {-144,-72,0,72,144} {
   \begin{scope}[shift={(5.5,3.5)},rotate=\x]
    \node (\x) at (0,-1) [fill,circle,inner sep=1pt] {};
   \end{scope}
}
\draw (a) node [above,shift={(0cm,0.1cm)}] {$0$};
\draw (144) node [above right] {$2$};
\draw (72) node [below right] {$3$};
\draw (0) node [below] {$4$};
\draw (-72) node [below left] {$5$};
\draw (-144) node [above left] {$1$};
\draw[thin,opacity=0.5] (a) to (-144);
\draw[thin, opacity=0.5] (a) to (144);
\draw[thin,opacity=0.5,out=-30,in=-150] (-72) to (72);
\draw[thin,opacity=0.5,out=-165,in=-45] (72) to (-0.3+5.5,-0.3+3.5);
\draw[thin,opacity=0.5,out=135,in=-90] (-0.3+5.5,-0.3+3.5) to (-144);

\draw[thin,opacity=0.5,out=-135,in=-80] (0.2+5.5,-0.2+3.5) to (-144);
\draw[thin,opacity=0.5,out=45,in=-100] (0.2+5.5,-0.2+3.5) to (144);

\draw (a) node [shift={(-2cm,0cm)}] {$\Pi$};

\draw[semithick,|->] (5.5,1.75) -- (5.5,1.25);

\node (a) at (2,0.75) [fill,circle,inner sep=1pt] {};
\node (b) at (7,0.75) [fill,circle,inner sep=1pt] {};

\draw (a) node [myblue,above,shift={(0.25cm,0cm)}] {$0^{(0)}$};
\draw (b) node [myblue,above,shift={(0.25cm,0cm)}] {$0^{(1)}$};
\draw (10.5,0.75) node [above,shift={(0.5 cm,0.1 cm)}] {$\dots$};
\draw (-0.5,0.75) node [above,shift={(-0.5 cm,0.1 cm)}] {$\dots$};
\draw[semithick] (-0.5,0.75) to (10.5,0.75);

\foreach \x in {0,1,2,3,4,5,6,7,8,9,10} {
   \begin{scope}[shift={(\x cm, 0 cm)}]
    \node (\x) at (0,-1) [fill,circle,inner sep=1pt] {};
   \end{scope}
}
\draw (0) node [myblue,below,shift={(0.25cm,0cm)}] {$1^{(0)}$};
\draw (1) node [myblue,below,shift={(0.25cm,0cm)}] {$2^{(0)}$};
\draw (2) node [myblue,below,shift={(0.25cm,0cm)}] {$3^{(0)}$};
\draw (3) node [myblue,below,shift={(0.25cm,0cm)}] {$4^{(0)}$};
\draw (4) node [myblue,below,shift={(0.25cm,0cm)}] {$5^{(0)}$};
\draw (5) node [myblue,below,shift={(0.25cm,0cm)}] {$1^{(1)}$};
\draw (6) node [below,shift={(0.25cm,0cm)}] {$2^{(1)}$};
\draw (7) node [below,shift={(0.25cm,0cm)}] {$3^{(1)}$};
\draw (8) node [below,shift={(0.25cm,0cm)}] {$4^{(1)}$};
\draw (9) node [below,shift={(0.25cm,0cm)}] {$5^{(1)}$};
\draw (10) node [below,shift={(0.25cm,0cm)}] {$1^{(2)}$};
\draw (10.5,-1) node [below,shift={(0.5 cm,-0.2 cm)}] {$\cdots$};
\draw (-0.5,-1) node [below,shift={(-0.50 cm,-0.2 cm)}] {$\cdots$};
\draw[semithick] (-0.5,-1) to (10.5,-1);

\draw[opacity=0,fill=myblue,fill opacity=0.15] (0,-1) -- (2,0.75) -- (7,0.75) -- (5,-1) -- cycle;
\node (a) at (2,0.75) [myblue,fill,circle,inner sep=1pt] {};
\node (b) at (7,0.75) [myblue,fill,circle,inner sep=1pt] {};
\foreach \x in {0,1,2,3,4,5} {
   \begin{scope}[shift={(\x cm, 0 cm)}]
    \node (\x) at (0,-1) [myblue,fill,circle,inner sep=1pt] {};
   \end{scope}
}
\draw (a) node [myblue,shift={(2.5cm,-0.75cm)}] {$\mathcal{P}_0$};

\draw[myblue,thin] (0) to (a);
\draw[myblue,thin] (1) to (a);
\draw[myblue,thin,out=55,in=125] (1) to (5);
\draw[myblue,thin,out=30,in=150] (2) to (4);
\draw[myblue,thin,out=45,in=135] (2) to (5);

\draw[myblue,thin] (5) to (b);
\draw[thin,opacity=0.5] (6) to (b);
\draw[thin,opacity=0.5,out=55,in=125] (6) to (10);
\draw[thin,opacity=0.5,out=30,in=150] (7) to (9);
\draw[thin,opacity=0.5,out=45,in=135] (7) to (10);

\draw[thin,opacity=0.5] (10) to (10.5,-0.6);
\draw[thin,opacity=0.5,out=-40,in=125] (-0.5,-0.48) to (0);
\draw[thin,opacity=0.5,out=-35,in=135] (-0.5,-0.6) to (0);

\node (0,0) [shift={(-2 cm,-0.125 cm)}]  {$\Phi(\Pi)$};
\end{tikzpicture}}
\caption{A triangulation of $S_5^1$ with associated $5$-periodic triangulation of $\mathcal{U}_5$.}\label{figexbijection}
\end{figure}
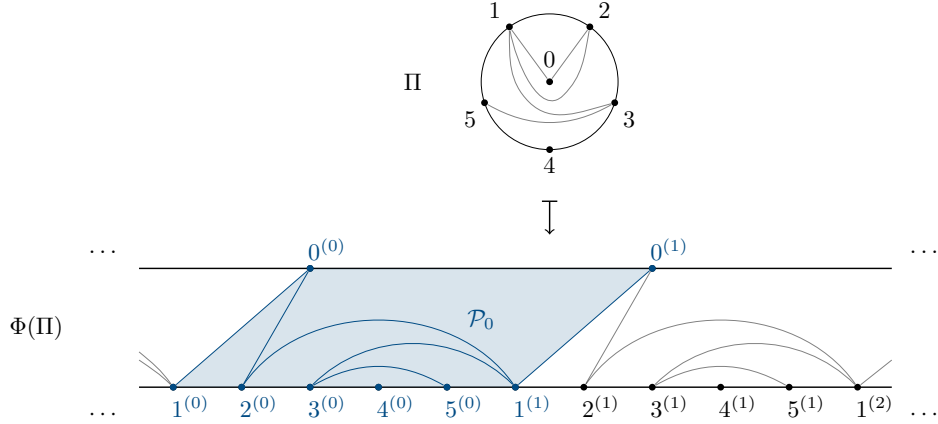

This is realized as follows: for every integer $k$ there is a natural embedding $\iota_k$ of the set $A(S_n^1)$ of arcs in $S_n^1$ into the set $A(\mathcal{U}_n)$ of arcs in $\mathcal{U}_n$, given by
 $$\iota_k(ij)=\begin{cases}
i^{(k)}j^{(k)}&\text{if }i<j,\\
i^{(k)}j^{(k+1)}&\text{if }i\ge j.
\end{cases}$$
We define the set $\Phi(\Pi)$ of arcs in $\mathcal{U}_n$ by the union of the disjoint sets $\iota_k(\Pi)$, $k\in\mathbb{Z}$, i.e.\
$$\Phi(\Pi)=\bigcup_{k\in\mathbb{Z}}\iota_k(\Pi).$$

Note that if an arc is contained in $\Phi(\Pi)$ then so are all its $n$-translates. Moreover, $\iota_k$ sends bridging arcs to bridging arcs. Thus $\Phi(\Pi)$ contains at least one bridging arc at $0^{(k)}$ for every integer $k$ since every triangulation of $S_n^1$ contains a bridging arc.  Clearly, the star-triangulation of $S_n^1$ gives the star-triangulation of $\mathcal{U}_n$. An other example is illustrated in Figure~\ref{figexbijection}.

\begin{prop}
Let $\Pi$ be a triangulation of $S_n^1$. Then $\mathcal{T}_n=\Phi(\Pi)$ is an $n$-periodic triangulation of $\mathcal{U}_n$.
\end{prop}

\begin{proof}
Clearly, if $\Pi=\Pi_\ast$ is the star-triangulation of $S_n^1$, then the claim is true. We proceed by induction on $n$. If $n=1$, there is only the star-triangulation. Now we assume the claim holds for $n\geq 1$. Let $\Pi$ be a triangulation of $S_{n+1}^1$ other than $\Pi_\ast$ with  bridging arc $0i$, $i\in\{1,2,\dots,n+1\}$, in $\Pi$. By Lemma~\ref{lemspecialmarkedpoint}, $\Pi$ contains a special marked point $x\in\{1,2,\dots,n+1\}\setminus\{i\}$. Then, by Corollary~\ref{corgluecuttriangle} and induction, $\Pi_{\setminus x}$ is a triangulation of $S_n^1$ providing an $n$-periodic triangulation $\mathcal{T}_n=\Phi(\Pi_{\setminus x})$ of $\mathcal{U}_n$. Recall, if $i\in\{1,2,\dots,x-1\}$, we set $i_x=i$ and for $i\in\{x+1,\dots, n+1\}$ we choose $i_x=i-1$. Then by construction, $\mathcal{T}_{n}$ contains $0^{(k)}{i_x}^{(k)}$ for every $k\in\mathbb{Z}$. We choose the fundamental domain for $\mathcal{T}_{n}$ given by the \mbox{$(n\!+\!3)$-gon} $\mathcal{P}'$ with vertices $0^{(0)},0^{(1)},{i_x}^{(1)}$, $({i_x}-1)^{(1)},\dots,1^{(1)},n^{(0)},$\linebreak ${(n-1)}^{(0)},\dots,{i_x}^{(0)}$ containing $n\!-\!1$ pairwise non-crossing arcs. Clearly, since $x$ is a special marked point in $\Pi$ it follows, for every $k$, that there is no arc in $\mathcal{T}_{n+1}=\Phi(\Pi)$ having $x^{(k)}$ as an endpoint. Then the union $\mathcal{P}$ of all triangles in $\mathcal{P}'$ together with the special triangle with vertices $x-1^{(0)},x^{(0)},x+1^{(0)}$ is an \mbox{$(n\!+\!4)$-gon} containing $n$ non-crossing arcs. Thus $\mathcal{P}$ provides a fundamental domain for $\mathcal{T}_{n+1}$ and it follows that $\mathcal{T}_{n+1}$ is an $(n+1)$-periodic triangulation of $\mathcal{U}_{n+1}$.
\end{proof}

Note that triangulations of $S_n^1$ with the same shape provide the same $n$-periodic triangulation of $\mathcal{U}_n$. Note also that if $x$ is a special marked point in $\Pi$, then $x^{(k)}$ is a special vertex in $\Phi(\Pi)$ for every integer $k$. Hence we have the following corollary.

\begin{cor}\label{corspecialvertex}
Let $\mathcal{T}_n$ be an $n$-periodic triangulation of $\mathcal{U}_n$ other than the star-triangulation with quiddity sequence $q_{\mathcal{T}_n}=(a_1,a_2,\dots,a_n)$. Then $a_x=1$ for some $x\in\{1,2,\dots,n\}$ and $x^{(k)}$ is a special vertex with respect to $\mathcal{T}_n$.
\end{cor}

\FloatBarrier
\subsection{From periodic triangulations of strips to triangulations of once-punctured discs}\label{secPhi}

Let $\mathcal{T}_n$ be an $n$-periodic triangulation of $\mathcal{U}_n$ with fundamental domain $\mathcal{P}$. We associate a triangulation $\Pi=\Psi(\mathcal{T}_n)$ of $S_n^1$ to $\mathcal{T}_n$ as follows: there is natural projection $\pi$ from the set $A(\mathcal{U}_n)$ of arcs in $\mathcal{U}_n$ to the set  $A(S_n^1)$ of arcs in $S_n^1$ given by
\begin{eqnarray*}
\pi\colon A(\mathcal{U}_n)& \to &A(S_n^1)\\
i^{(k)}j^{(l)}&\mapsto &ij
\end{eqnarray*}
(with $l\in \{k,k+1\}$ cf.\ Lemma~\ref{lemarcinfundamentaldomain}). We define the set $\Pi$ of arcs in $S_n^1$ to be the image of $\mathcal{P}$ under $\pi$, i.e.\
$$\Psi(\mathcal{T}_n)=\bigcup_{i^{(k)}j^{(l)}\in\mathcal{T}_n\cap\mathcal{P}} \pi\left(i^{(k)}j^{(l)}\right).$$

One can easily convince oneself that the definition of $\Psi$ does not depend on the choice of a fundamental domain. Note that $\pi$ sends bridging arcs to bridging arcs. Thus by Remark~\ref{remfundamentaldomainstrip} $\Psi(\mathcal{T}_n)$ contains at least one bridging arc. Note also that a special vertex in $\mathcal{T}_n$ maps to a special marked point in $\Psi(\mathcal{T}_n)$.

Clearly, the star-triangulation of $\mathcal{U}_n$ leads to the star-triangulation of $S_n^1$. Another example for $n=5$, Figure~\ref{figextriangulationstrip5p1} shows the periodic triangulation of the strip whose image under $\Psi$ is the triangulation of the once-punctured disc given in Figure~\ref{figexquidditysequencedisc}. The proof of the next result is straightforward.

\begin{prop}
Let $\mathcal{T}_n$ be a $n$-periodic triangulation of $\mathcal{U}_n$. Then $\Pi=\Psi(\mathcal{T}_n)$ is a triangulation of $S_n^1$.
\end{prop}

\subsection{Bijection between triangulations of once-punctured discs and periodic triangulations of strips}

In the sequel $\Phi$ and $\Psi$ denote the maps defined respectively in Subsections \ref{secPsi} and \ref{secPhi}.

\begin{thm}\label{thmbijection}
The maps $\Phi$ and $\Psi$ are inverse bijections between triangulations of $S_n^1$ and \mbox{$n$-periodic} triangulations of $\mathcal{U}_n$.
\end{thm}

\begin{proof}
We have to proof that $\Phi\circ\Psi$ and $\Psi\circ\Phi$ are the the identity on $\mathcal{U}_n$ and on $S_n^1$, respectively.
\begin{inparaenum}[(i)] 
\item Let $\mathcal{T}_n$ be an $n$-periodic triangulation of $\mathcal{U}_n$. We write $\Pi=\Psi(\mathcal{T}_n)$, $\mathcal{T}_n'=\Phi(\Pi)$ and show that $\mathcal{T}_n=\mathcal{T}'_n$. First, we choose an arbitrary arc $\gamma=i^{(k)}j^{(l)}$ in $\mathcal{T}_n$. We have $\pi(\gamma)=ij\in\Pi$. By Lemma~\ref{lemarcinfundamentaldomain}, we know that $\gamma=i^{(k)}j^{(k)}$ with $i<j$, or $\gamma=i^{(k)}j^{(k+1)}$ with $i\geq j$. If $i<j$, then $\iota_k(ij)=i^{(k)}j^{(k)}\in \mathcal{T}'_n$, otherwise, if $i\geq j$, then $\iota_k(ij)=i^{(k)}j^{(k+1)}\in \mathcal{T}'_n$. In both cases $\iota_k(ij)=\iota_k(\pi(\gamma))=\gamma \in \mathcal{T}'_n$. Next, we show that $\mathcal{T}_n'$ is contained in $\mathcal{T}_n$. Let $\gamma'={i}^{(k')}{j}^{(l')}$ be an arbitrary arc in $\mathcal{T}'_n$. Either $i<j$, in which case $\gamma'=i^{(k')}j^{(k')}$, or $i\geq j$, in which case $\gamma'=i^{(k')}j^{(k'+1)}$ (Lemma~\ref{lemarcinfundamentaldomain}). Since $\mathcal{T}'_n$ is the image of $\Pi$, we have $\gamma'={i}^{(k')}{j}^{(l')}\in \iota_{k'}(\Pi)$, thus $ij\in\Pi$. Since $\Pi=\Psi(\mathcal{T}_n)$ there exists $\gamma\in \mathcal{T}_n$ with $\pi(\gamma)=ij$. Clearly, $\gamma=i^{(k)}j^{(l)}$ for some $k,l\in\mathbb{Z}$, where again either $\gamma=i^{(k)}j^{(k)}\in \mathcal{T}_n$ with $i<j$, or $\gamma=i^{(k)}j^{(k+1)}\in \mathcal{T}_n$ with $i\geq j$. Hence $\gamma'$ is an $n$-translate of $\gamma$, so $\gamma'\in\mathcal{T}_n$.

\item We consider a triangulation $\Pi$ of $S_n^1$ and write $\mathcal{T}_n=\Phi(\Pi)$, $\Pi'=\Psi(\mathcal{T}_n)$. Let $\gamma=ij$ be an arc in $\Pi$. For some $k\in\mathbb{Z}$ we have either $\iota_k(\gamma)=i^{(k)}j^{(k)}\in\mathcal{T}_n$ if $i<j$, or $\iota_k(\gamma)=i^{(k)}j^{(k+1)}\in\mathcal{T}_n$ otherwise. In both cases we get $\pi(\iota_k(\gamma))=ij\in \Pi'$, hence $\gamma\in \Pi'$. Since this is true for all $n$ non-crossing arcs in $\Pi$ we have $\Pi'=\Pi$ as desired. This establishes the bijection.
 \end{inparaenum}
\end{proof}

Before moving on to matching numbers, let us summarize some immediate facts about periodic triangulations of strips implied by the bijection of Theorem~\ref{thmbijection} and our previous work on triangulations of once-punctured discs. Especially, how periodic triangulations of strips are linked to arithmetic friezes.

\begin{lem}\label{lemequivquid}
Let $\mathcal{T}_n$ be an $n$-periodic triangulation of $\mathcal{U}_n$ with quiddity sequence $q_{\mathcal{T}_n}$ and let  \mbox{$\Pi=\Psi(\mathcal{T}_n)$} be the associated triangulation of $S_n^1$ with quiddity sequence $q_\Pi$. Then $q_{\mathcal{T}_n}$ equals $q_\Pi$. In particular, $\mathcal{T}_n$ provides an $n$-arithmetic frieze $\mathcal{F}_{\mathcal{T}_n}$ that is equal to $\mathcal{F}_{\Pi}$.
\end{lem}

Exactly as for triangulations of once-punctured discs we can use the operations of cutting and gluing triangles for fundamental domains of periodic triangulations of strips. Here, we only allow adding and removing triangles at the lower boundary. 

\begin{cor}\label{corgluecutstrip}
Let $\mathcal{T}_n$ be an $n$-periodic triangulation of $\mathcal{U}_n$ with quiddity sequence\linebreak $q_{\mathcal{T}_n}=(a_1,a_2,\dots,a_n)$.

\begin{inparaenum}[$a)$] \label{corgluecutstripa}
\item \label{corgluecutstripa} Assume that $a_x=1$ for some $x\in\{1,2,\dots,n\}.$ Let $\mathcal{T}_{n\setminus x}$ denote the union of all triangles in $\mathcal{T}_n$ other than the special triangles at $x^{(k)}$, for all $k\in\mathbb{Z}$. Then $\mathcal{T}_{n\setminus x}$ is an $(n\!-\!1)$-periodic triangulation of $\mathcal{U}_{n-1}$ with quiddity sequence
$$q_{\mathcal{T}_{n\setminus x}}=\begin{cases}
(2)&\text{if } n=2,\\
(a_1,\dots,a_{x-2},a_{x-1}\!-\!1,a_{x+1}\!-\!1,a_{x+2},\dots,a_n)&\text{otherwise}.
\end{cases}$$

\item \label{corgluecutstripb} For some $i\in\{1,2,\dots,n\}$ insert a vertex $x^{(k)}$ between $i^{(k)}$ and $(i+1)^{(k)}$ together with the arc $i^{(k)}(i+1)^{(k)}$  for every $k\in\mathbb{Z}$. Let $\mathcal{T}_{n\cup x}$ denote the union of all triangles in $\mathcal{T}_n$ together with the triangles having vertices $i^{(k)},x^{(k)}$ and $(i+1)^{(k)}$, for all $k\in\mathbb{Z}$. Then $\mathcal{T}_{n\cup x}$ is an $(n\!+\!1)$-periodic triangulation of $\mathcal{U}_{n+1}$ with quiddity sequence
$$q_{\mathcal{T}_{n\cup x}}=\begin{cases}
(4,1)&\text{if } n=1,\\
(a_{1^{(0)}},\dots,a_{i-1},a_i\!+\!1,1,a_{i+1}\!+\!1,a_{i+2},\dots,a_n)&\text{otherwise}.
\end{cases}$$
\end{inparaenum}
\end{cor}

\begin{rem}
By Corollary~\ref{corgluecutstrip} $\ref{corgluecutstripb})$, starting with a fundamental domain for the star-triangulation of $\mathcal{U}_n$ and iteratively inserting triangles provides all $m$-periodic triangulations of $\mathcal{U}_m$ ($m\ge n$) with $n$ bridging arcs at a vertex on the upper boundary.
\end{rem}

Of course, we also have the following result.

\begin{cor}\label{corspecialvertexfriezeequivalce}
Given an $(n\!+\!1)$-periodic triangulation $\mathcal{T}_{n+1}$ of $\mathcal{U}_{n+1}$ with special vertex $x^{(k)}$. Let $\mathcal{F}_{\mathcal{T}_{n+1}}$ be the $(n\!+\!1)$-arithmetic frieze associated to $\mathcal{T}_{n+1}$ and let $\mathcal{F}_{\mathcal{T}_{n+1\setminus x}}$ be the $n$-arithmetic frieze associated to $\mathcal{T}_{n+1\setminus x}$. Then $\mathcal{F}_{\mathcal{T}_{n+1\setminus x}}$ is obtained from $\mathcal{F}_{\mathcal{T}_{n+1}}$ by $(n\!+\!1)$-cutting above $a_{x}$, and vice versa $\mathcal{F}_{\mathcal{T}_{n+1}}$ is obtained from $\mathcal{F}_{\mathcal{T}_{n+1\setminus x}}$ by $n$-gluing above $(a_{x-1}-1,a_{x+1}-1)$.
\end{cor}

\subsection{Matching numbers}

Given a triangulation of a once-punctured disc we will consider the associated periodic triangulation of the corresponding strip. We are interested in the ways to allocate triangles in a periodic triangulation of a strip to sets of consecutive vertices on the lower boundary. 

\begin{defin}
Let $\mathcal{T}_n$ be an $n$-periodic triangulation of $\mathcal{U}_n$ and let $I$ be a set of $s\geq 1$ consecutive vertices $v_1,v_2\dots,v_s$ on the lower boundary of $\mathcal{U}_n$ ($v_1\le v_s$). A \emph{matching} between $I$ and $\mathcal{T}_n$ is an $s$-tuple $(\tau_1,\tau_2, \dots,\tau_s)$ of pairwise distinct triangles in $\mathcal{T}_n$ such that $\tau_i$ is incident with $v_i$. $\mathcal{M}_{v_1v_s}=\mathcal{M}_{v_1v_s}(\mathcal{T}_n)$ is the \emph{set of all matchings} between $I$ and $\mathcal{T}_n$.
\end{defin}

\begin{rem}\label{remmatchingsinglevertex}
Clearly, the number of matchings between a set $I=\{v\}$ containing one single vertex $v$ and an $n$-periodic triangulation $\mathcal{T}_n$ of $\mathcal{U}_n$ equals the number of triangles incident with $p$. In other words, if $q_{\mathcal{T}_n}=(a_1,a_2,\dots,a_n)$ is the quiddity sequence of $\mathcal{T}_n$, it follows, by definition, that $\abs{\mathcal{M}_{i^{(k)}i^{(k)}}}=a_i$. In particular, $\abs{\mathcal{M}_{i^{(k)}i^{(k)}}}=1$ if $i^{(k)}$ is special. Moreover,
for two vertices $i^{(k)}\leq j^{(l)}$ on the lower boundary of $\mathcal{U}_n$, we have $\abs{\mathcal{M}_{i^{(k)}j^{(l)}}}=\abs{\mathcal{M}_{i^{(k+m)}j^{(l+m)}}}$ for all $m\in\mathbb{Z}$.
\end{rem}

Considering two periodic triangulations of strips where one is obtained from the other by gluing triangles as in Corollary~\ref{corgluecutstrip}~\ref{corgluecutstripb}), we would like to know how the matching numbers for the two periodic triangulations are related to each other. The following lemma provides the answer.

Such as in Corollary~\ref{cornglueperiodicfrieze}, we shall use the following notation. For $i,x\in\mathbb{Z}$, we set $i_x=i-t$ whenever $x+(t-1)(n+1)< i\le x+t(n+1)$.

\begin{lem}\label{lemmatchingsnglue}
Let $\mathcal{T}_{n+1}$ be an $(n\!+\!1)$-periodic triangulation of $\mathcal{U}_{n+1}$ with quiddity sequence\linebreak $q_{\mathcal{T}_{n+1}}=(a_1,a_2,\dots,a_{n+1})$ and assume that $a_x=1$ for some $x\in\{1,2,\dots,n+1\}$. Let $\mathcal{T}_{n}=\mathcal{T}_{n+1\setminus x }$. Then for two vertices $i^{(k)}< j^{(l)}$ on the lower boundary of $\mathcal{U}_{n+1}$
$$\abs{\mathcal{M}_{i^{(k)}j^{(l)}}(\mathcal{T}_{n+1})}=\begin{cases}
\abs{\mathcal{M}_{i_{x+1}^{(k)}j_{x-1}^{(l)}}(\mathcal{T}_n)} & i\not\equiv x+1,j\not\equiv x-1,\\[0.25cm]
\abs{\mathcal{M}_{i_{x+1}^{(k)}j_{x-1}^{(l)}}(\mathcal{T}_n)} +\abs{\mathcal{M}_{i_{x+1}^{(k)}(j_{x-1}-1)^{(l)}}(\mathcal{T}_n)}& i\not \equiv x+1,j\equiv x-1,\\[0.25cm]
 \abs{\mathcal{M}_{(i_{x+1}-1)^{(k)}j_{x-1}^{(l)}}(\mathcal{T}_n)}+\abs{\mathcal{M}_{i_{x+1}^{(k)}j_{x-1}^{(l)}}(\mathcal{T}_n)} &i\equiv x+1,j\not\equiv x-1,\\[0.25cm]
\abs{\mathcal{M}_{(i_{x+1}-1)^{(k)}j_{x-1}^{(l)}}(\mathcal{T}_n)} +\abs{\mathcal{M}_{(i_{x+1}-1)^{(k)}{(j_{x-1}-1)}^{(l)}}(\mathcal{T}_n)}\\[0.25cm]
+\abs{\mathcal{M}_{i_{x+1}^{(k)}j_{x-1}^{(l)}}(\mathcal{T}_n)} +\abs{\mathcal{M}_{i_{x+1}^{(k)}{(j_{x-1}-1)}^{(l)}}(\mathcal{T}_n)} & i\equiv x+1,j\equiv x-1,
\end{cases}$$
where $\equiv$ stands for equality modulo $n+1$.
\end{lem}

The proof of Lemma~\ref{lemmatchingsnglue} is a tedious but straightforward case-by-case study. We thus omit it. Note that for $x=n+1$, we actually get $x+1=n+2$, but $1\le i\le n+1$. Similarly, if $x=1$, in this case we have $x-1=0$, but $1\le j\le n+1$. This is why we consider the equality up to congruence classes reduced modulo $n+1$. 

\begin{thm}\label{thmmatchings}
Let $\mathcal{T}_n$ be an $n$-periodic triangulation of $\mathcal{U}_n$ with associated $n$-arithmetic frieze\linebreak $\mathcal{F}_{\mathcal{T}_n}=(m_{ij})_{j-i\ge -2}$. Then for $i\le j$
$$m_{ij}=\abs{\mathcal{M}_{i^{(k)}j^{(k)}}}.$$
\end{thm}

\begin{proof}
W.l.o.g.\ we may take $k=0$. We first prove the claim for star-triangulations. Let $\mathcal{T}_n=\mathcal{T}_\ast$ be the star-triangulation of $\mathcal{U}_n$ with associated basic infinite frieze $\mathcal{F}_\ast=(m_{ij})_{j-i\ge -2}$, where $m_{ij}=j-i+2$, see Figures~\ref{figbstartriangulationstrip} and \ref{figbasicfrieze}. Clearly, $\abs{\mathcal{M}_{i^{(k)}i^{(k)}}}=2$ and for two vertices $i^{(k)}\leq j^{(l)}$ on the lower boundary of $\mathcal{U}_n$, we have $\abs{\mathcal{M}_{i^{(k)}(j+1)^{(l)}}}=\abs{\mathcal{M}_{i^{(k)}j^{(l)}}}+1$. Let $i \le j$, so $\abs{\mathcal{M}_{i^{(0)}j^{(0)}}}=j-i+2=m_{ij}$. Hence the result follows for $\mathcal{T}_\ast$, in particular, for $n=1$.

Now we use induction on $n$ to prove the claim for the remaining periodic triangulations. We assume that the result holds for any $n$-periodic triangulation of $\mathcal{U}_n$. We consider an $(n\!+\!1)$-periodic triangulation $\mathcal{T}_{n+1}$ of $\mathcal{U}_{n+1}$, $\mathcal{T}_{n+1}\ne\mathcal{T}_\ast$, with quiddity sequence $q_{\mathcal{T}_{n+1}}=(a_1,a_2,\dots,a_{n+1})$ and associated $(n\!+\!1)$-arithmetic frieze $\mathcal{F}_{\mathcal{T}_{n+1}}=(m_{ij})_{j-i\ge -2}$. For $i=j$, Remark~\ref{remmatchingsinglevertex} already gives the desired result, so we assume $i< j$.

Corollary~\ref{corspecialvertex} implies that $\mathcal{T}_{n+1}$ has special vertices, so there exists $x\in\{1,2,\dots n+1\}$ such that $a_x=1$, and $\mathcal{T}_{n}:=\mathcal{T}_{n+1\setminus x }$ is an $n$-periodic triangulation of $\mathcal{U}_{n}$ with quiddity sequence $q_{\mathcal{T}_n}=(\check a_1,\check a_2,\dots,\check a_n)$ as given in Corollary~\ref{corgluecutstrip}. We write $\mathcal{F}_{\mathcal{T}_n}=(\check m_{ij})_{j-i\ge -2}$ for the $n$-arithmetic frieze  associated to $\mathcal{T}_{n}$, and we obtain $\mathcal{F}_{\mathcal{T}_{n+1}}$ from $\mathcal{F}_{\mathcal{T}_n}$ by $n$-gluing above $(\check a_{x-1},\check a_x)$ (Corollaries \ref{corgluecutstrip} \ref{corgluecutstripa}) and \ref{corspecialvertexfriezeequivalce}).

It is enough to prove the claim for a fundamental domain for $\mathcal{F}_{\mathcal{T}_{n+1}}$, so we may consider $m_{ij}$ with $i\in\{1,2,\dots,n+1\}$ and $j\ge i+1$. We distinguish four separate cases and show for each case the result by making use of Corollary~\ref{cornglueperiodicfrieze} (for $k=x-1$), the inductive hypothesis, and Lemma~\ref{lemmatchingsnglue}. Let $i\not\equiv x+1$ and $j\not\equiv x-1$ (reduced modulo $n+1$). In this case we have $m_{ij}=\check m_{i_{x+1}j_{x-1}}=\abs{\mathcal{M}_{i_{x+1}^{(0)}j_{x-1}^{(0)}}(\mathcal{T}_n)}=\abs{\mathcal{M}_{i^{(0)}j^{(0)}}(\mathcal{T}_{n+1})}$. If $i\not \equiv x+1$ and $j\equiv x-1$,  we get $m_{ij}=\check m_{i_{x+1},j_{x-1}-1}+\check m_{i_{x+1}j_{x-1}} =\abs{\mathcal{M}_{i_{x+1}^{(0)}{(j_{x-1}-1)}^{(0)}}(\mathcal{T}_n)}+\abs{\mathcal{M}_{i_{x+1}^{(0)}j_{x-1}^{(0)}}(\mathcal{T}_n)}=\abs{\mathcal{M}_{i^{(0)}j^{(0)}}(\mathcal{T}_{n+1})}$. Suppose $i\equiv x+1$ and $j\not \equiv x-1$. It follows analogously that $m_{ij}=\check m_{i_{x+1}-1,j_{x-1}}+\check m_{i_{x+1},j_{x-1}}=\abs{\mathcal{M}_{(i_{x+1}-1)^{(0)}j_{x-1}^{(0)}}(\mathcal{T}_n)}+\abs{\mathcal{M}_{i_{x+1}^{(0)}j_{x-1}^{(0)}}(\mathcal{T}_n)}=\abs{\mathcal{M}_{i^{(0)}j^{(0)}}(\mathcal{T}_{n+1})}$. Finally, if $i\equiv x+1$ and $j\equiv x-1$, we have $m_{ij}=\check m_{i_{x+1}-1,j_{x-1}-1}+\check m_{i_{x+1}j_{x-1}}+\check m_{i_{x+1}-1,j_{x-1}}+\check m_{i_{x+1},j_{x-1}-1}=\abs{\mathcal{M}_{(i_{x+1}-1)^{(0)}{(j_{x-1}-1)}^{(0)}}(\mathcal{T}_n)}+\abs{\mathcal{M}_{i_{x+1}^{(0)}j_{x-1}^{(0)}}(\mathcal{T}_n)}+\abs{\mathcal{M}_{(i_{x+1}-1)^{(0)}j_{x-1}^{(0)}}(\mathcal{T}_n)}+\abs{\mathcal{M}_{i_{x+1}^{(0)}{(j_{x-1}-1)}^{(0)}}(\mathcal{T}_n)}=\abs{\mathcal{M}_{i^{(0)}j^{(0)}}(\mathcal{T}_{n+1})}.$ This completes the proof.
\end{proof}

To illustrate Theorem~\ref{thmmatchings} we give the following example.

\begin{ex}
We consider the $5$-periodic triangulation of $\mathcal{U}_5$ with associated triangulation of $S_5^1$ given in Figure~\ref{figexbijection}. Then the associated $5$-arithmetic frieze $\mathcal{F}_5=(m_{ij})_{j-i\ge -2}$ has quiddity sequence $q_{\mathcal{F}_5}=(4,3,3,1,2)$ and the fundamental domain $\mathcal{D}=(m_{ij})_{1\leq i\leq 5, j\ge i-2}$ looks as follows
\begin{center}
\scalebox{0.9}{\begin{tikzpicture}[font=\normalsize] 
  \matrix(m) [matrix of math nodes,row sep={1.5em,between origins},column sep={1.5em,between origins},nodes in empty cells]{
&&&&&&&&&&&&&&&&&&&\\
&0&&0&&0&&0&&0&&&&&&&&&&\\
&&1&&1&&1&&1&&1&&&&&&&&&\\
&&&4&&3&&3&&1&&2&&&&&&&&\\
&&&&11&&8&&2&&1&&7&&&&&&&\\
&&&&&29&&5&&1&&3&&19&&&&&&\\
&&&&&&18&&2&&2&&8&&50&&&&&\\
&&&&&&&7&&3&&5&&21&&31&&&&\\
&&&&&&&&10&&7&&13&&13&&12&&&\\
&&&&&&&&&23&&18&&8&&5&&17&&\\[+1mm]
&&&&&&&&&&&&\node[rotate=-6.5,shift={(-0.034cm,-0.08cm)}]  {\ddots};&&&&\node[rotate=-6.5,shift={(-0.034cm,-0.08cm)}]  {\ddots};&&&\\
};

\draw[opacity=0,rounded corners, fill=myblue,fill opacity=0.15] (m-1-1.south west) -- (m-10-10.south west) --($(m-10-20.south west)+(-0.25cm,0cm)$) -- ($(m-1-11.south west)+(-0.25cm,0cm)$)-- cycle;
  
\end{tikzpicture}}
\end{center}

Using matchings between sets of $s\ge 1$ consecutive vertices, starting at $i^{(0)}$, and $\mathcal{T}_5$ we obtain the matching numbers $\abs{\mathcal{M}_{i^{(0)}(i+s-1)^{(0)}]}}$ given in the table below. For fixed $i$, this provides exactly the non-trivial entries in the {\sc se}-diagonals of $\mathcal{F}_5$ whereas we get the non-trivial rows of $\mathcal{F}_5$ if $s$ is fixed.\vspace{0.25cm}

\begin{center}
\begin{tabular}{|l|*{7}{c|}}\hline
\diaghead(-1,1){\hskip1em }%
{$i$}{$s$}&\thead{$1$}&\thead{$2$}&\thead{$3$}&\thead{$4$}&\thead{$5$}&\thead{$6$}&\thead{$7$}\\  \hline
\thead{$1$}& $4$ & $11$& $29$ & $18$& $7$& $10$& $23$ \\    \hline
\thead{$2$}& $3$ & $8$& $5$ & $2$& $3$ & $7$& $18$ \\    \hline
\thead{$3$}& $3$ & $2$&$1$&$2$&$5$& $13$& $8$\\    \hline
\thead{$4$}& $1$ & $1$&$3$&$8$&$21$& $13$& $5$\\    \hline
\thead{$5$}& $2$ & $7$&$19$&$50$&$31$& $12$& $17$\\    \hline
\end{tabular}\vspace{0.25cm}
\end{center}
\end{ex}

%
\section{An alternative description}\label{seclabelingalgorithm}
%

In this section we give an alternative description of the entries in arithmetic friezes associated to triangulated once-punctured discs. This will provide, in a simple way, all diagonals of an arithmetic frieze. Conway and Coxeter in (32) of \cite{CCI} allocated non-negative numbers to the vertices of triangulated polygons and these numbers appear in the diagonals of the associated finite frieze. Adapting their strategy we assign numbers to the vertices of periodic triangulations of strips in an iterative way. Moreover, it turns out that this labeling algorithm can be used to determine the common differences of the arithmetic progressions that appear in an arithmetic frieze.

Let $\mathcal{T}_ n$ be an $n$-periodic triangulation of $\mathcal{U}_n$. We consider a fixed vertex $v$ on the lower boundary of $\mathcal{U}_n$ called \emph{starting vertex} and attach labels $n_v\!\left(w\right)\in\mathbb{Z}_{\geq 0}$ to every vertex $w$ of $\mathcal{U}_n$ as follows: we first set $n_{v}\!\left(v\right)=0$, and $n_v\!\left(w\right)=1$ whenever the vertex $w$ is joint with $v$ by a boundary segment or an arc in $\mathcal{T}_n$. Note that $w$ may lie on the upper boundary. As soon as a label is given for a vertex $0^{(k)}$ on the upper boundary of $\mathcal{U}_n$ we label all the remaining vertices on the upper boundary by the same number, i.e.\ we then set $n_v\!\left(0^{(k')}\right):=n_v\!\left(0^{(k)}\right)$ for all $k'\ne k.$

Once all neighbors (through arcs or boundary segments) of $v$ have obtained their label, we iteratively define the labels for the remaining vertices on the lower boundary. Whenever there is a triangle in $\mathcal{T}_n$ given by three vertices $w_1,w_2$ and $w_3$ on the lower boundary such that two of its vertices already have a label, e.g.\ $w_1$ and $w_2$, we take their sum for the label of the remaining vertex: $n_v\!\left(w_3\right):=n_v\!\left(w_1\right)+n_v\!\left(w_2\right).$

If there is no such a triangle in $\mathcal{T}_n$ left, we consider triangles in $\mathcal{T}_n$ of one of the following types
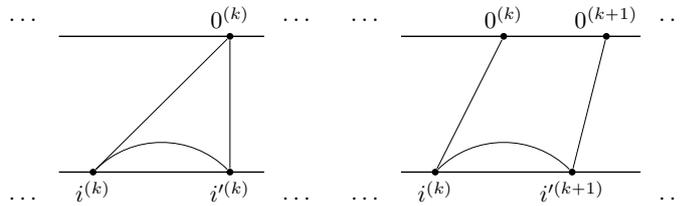
\begin{figure}[h]
\scalebox{0.9}{\begin{tikzpicture}[scale=1,font=\normalsize]

\node (a) at (2,1) [fill,circle,inner sep=1pt] {};
\node (b) at (6,1) [fill,circle,inner sep=1pt] {};
\node (c) at (7.5,1) [fill,circle,inner sep=1pt] {};

\draw (a) node [above] {$0^{(k)}$};
\draw (b) node [above] {$0^{(k)}$};
\draw (c) node [above] {$0^{(k+1)}$};

\draw (8,1) node [above,shift={(0.5 cm,0.1 cm)}] {$\dots$};
\draw (-0.5,1) node [above,shift={(-0.5 cm,0.1 cm)}] {$\dots$};

\draw (2.5,1) node [above,shift={(0.5 cm,0.1 cm)}] {$\dots$};
\draw (4.5,1) node [above,shift={(-0.5 cm,0.1 cm)}] {$\dots$};

\draw[semithick] (-0.5,1) to (2.5,1);
\draw[semithick] (4.5,1) to (8,1);

\foreach \x in {0,2,5,7} {
   \begin{scope}[shift={(\x cm, 0 cm)}]
    \node (\x) at (0,-1) [fill,circle,inner sep=1pt] {};
   \end{scope}
}
\draw (0) node [below] {$i^{(k)}$};
\draw (2) node [below] {$i'^{(k)}$};
\draw (5) node [below] {$i^{(k)}$};
\draw (7) node [below] {$i'^{(k+1)}$};

\draw (8,-1) node [below,shift={(0.5 cm,-0.2 cm)}] {$\cdots$};
\draw (-0.5,-1) node [below,shift={(-0.50 cm,-0.2 cm)}] {$\cdots$};
\draw (2.5,-1) node [below,shift={(0.5 cm,-0.2 cm)}] {$\cdots$};
\draw (4.5,-1) node [below,shift={(-0.50 cm,-0.2 cm)}] {$\cdots$};

\draw[semithick] (-0.5,-1) to (2.5,-1);
\draw[semithick] (4.5,-1) to (8,-1);

\draw[thin] (0) to (a);
\draw[thin] (2) to (a);
\draw[thin,out=45,in=135] (0) to (2);
\draw[thin] (5) to (b);
\draw[thin] (7) to (c);
\draw[thin,out=45,in=135] (5) to (7);
\end{tikzpicture}}
\end{figure}

\noindent where two (or three) vertices already have a label. For the triangle on the left we use the same rule as before, and take the sum of the two given labels for the remaining label. In the triangle on the right we consider $0^{(k)}$ and $0^{(k+1)}$ to be a single vertex with label $n_v\!\left(0^{(k)}\right)$ and also take the sum of the two given labels as the remaining label.

Note that either $0^{(k)}$ and $0^{(k+1)}$ already have the same label or they will get the same label now. As mentioned before once a vertex on the upper boundary has a label, all vertices on the upper boundary obtain the same label. 

Finally, we continue the labeling for the remaining vertices on the lower boundary as done before: if the labels are given for two vertices in a triangle in $\mathcal{T}_n$, we take their sum for the remaining one taking care that whenever a triangle is a four-sided region, we only take one of the labels on the upper boundary for the sum.

\begin{figure}[t]
\scalebox{0.9}{\begin{tikzpicture}[font=\normalsize]

\draw (12.5,1.25) node [above,shift={(0.5cm,0cm)}] {$\cdots$};
\draw (-1.5,1.25) node [above,shift={(-0.5 cm,0cm)}] {$\cdots$};
\draw[semithick] (-1.5,1.25) to (12.5,1.25);

\node (a) at (-1,1.25) [fill,myblue,circle,inner sep=1pt] {};
\node (b) at (4,1.25) [fill,myblue,circle,inner sep=1pt] {};
\node (c) at (9,1.25) [fill,circle,inner sep=1pt] {};

\draw (a) node [myblue,above] {$2$};
\draw (b) node [myblue,above] {$2$};
\draw (c) node [above] {$2$};

\foreach \x in {-1,0,1,2,3,4,5,6,7,8,9,10,11,12} {
   \begin{scope}[shift={(\x cm, 0 cm)}]
    \node (\x) at (0,-1) [fill,circle,inner sep=1pt] {};
   \end{scope}
}

\draw (12.5,-1) node [below,shift={(0.5cm,-0.2cm)}] {$\cdots$};
\draw (-1.5,-1) node [below,shift={(-0.5 cm,-0.2 cm)}] {$\cdots$};
\draw[semithick] (-1.5,-1) to (12.5,-1);

\draw (-1) node [myblue,below,shift={(0.25cm,0cm)}] {$5$};
\draw (0) node [myblue,below] {$1$};
\draw (1) node [myblue,below] {$1$};
\draw (2) node [myblue,below] {$\begin{matrix} n_{2^{(0)}}\! \left(2^{(0)}\right)\\[0.5mm]\,\veq\\[-1mm]0\end{matrix}$};
\draw (3) node [myblue,below] {$1$};
\draw (4) node [myblue,below] {$1$};
\draw (5) node [myblue,below] {$1$};
\draw (6) node [below] {$5$};
\draw (7) node [below] {$4$};
\draw (8) node [below] {$11$};
\draw (9) node [below] {$7$};
\draw (10) node [below] {$3$};
\draw (11) node [below] {$11$};
\draw (12) node [below,shift={(0.25cm,0cm)}] {$2^{(2)}$};

\draw[opacity=0, fill=myblue,fill opacity=0.15] (0,-1) -- (-1,1.25) -- (4,1.25) -- (5,-1) -- cycle;
\foreach \x in {0,1,2,3,4,5} {
   \begin{scope}[shift={(\x cm, 0 cm)}]
    \node (\x) at (0,-1) [myblue,fill,circle,inner sep=1pt] {};
   \end{scope}
}
\node (a) at (-1,1.25) [fill,myblue,circle,inner sep=1pt] {};
\node (b) at (4,1.25) [fill,myblue,circle,inner sep=1pt] {};

\draw (a) node [myblue,shift={(2.5cm,-0.5cm)}] {$\mathcal{P}$};

\draw[myblue,thin] (0) to  (a);
\draw[myblue!,thin,out=42.5,in=135] (0) to  (5);
\draw[myblue,thin,out=30,in=150] (0) to  (2);
\draw[myblue,thin,out=35,in=145] (2) to (5);
\draw[myblue,thin,out=30,in=150] (2) to  (4);

\draw[myblue,thin] (5) to  (b);
\draw[thin,opacity=0.5,out=42.5,in=135] (5) to (10);
\draw[thin,opacity=0.5,out=30,in=150] (5) to  (7);
\draw[thin,opacity=0.5,out=35,in=145] (7) to (10);
\draw[thin,opacity=0.5,out=30,in=150] (7) to (9);

\draw[thin,opacity=0.5] (10) to (c);
\draw[thin,opacity=0.5,out=30,in=150] (10) to  (12);
\draw[thin,opacity=0.5,out=42.5,in=180] (10) to (12.5,0);
\draw[thin,opacity=0.5,out=35,in=200] (12) to (12.5,-0.72);
\draw[thin,opacity=0.5,out=30,in=195] (12) to (12.5,-0.77);

\draw[thin,opacity=0.5,out=-15,in=135] (-1.5,-0.1) to (0);
\draw[thin,opacity=0.5,out=0,in=150] (-1.5,-0.48) to (0);
\draw[thin,opacity=0.5,out=-15,in=150] (-1.5,-0.77) to (-1);

\node at (-0.5,-2) [opacity=0.5,align=left]  {$\Leftarrow$  };
\node at (-0.5,-2.5) [opacity=0.5,align=left]  {{\sc sw}-diagonal through $a_1=1$};
\node at (4.5,-2) [opacity=0.5,align=left]   {$\Rightarrow$};
\node at (4.5,-2.5) [opacity=0.5,align=left]   {{\sc se}-diagonal through $a_3=1$};

\end{tikzpicture}}
\caption{For starting vertex $2^{(0)}$, the labels attached to the $5$-periodic triangulation of $\mathcal{U}_5$ in Figure~\ref{figextriangulationstrip5p1} with associated $5$-arithmetic frieze given in Figure~\ref{figexfrieze5p2}.}\label{figexlabel}
\end{figure}
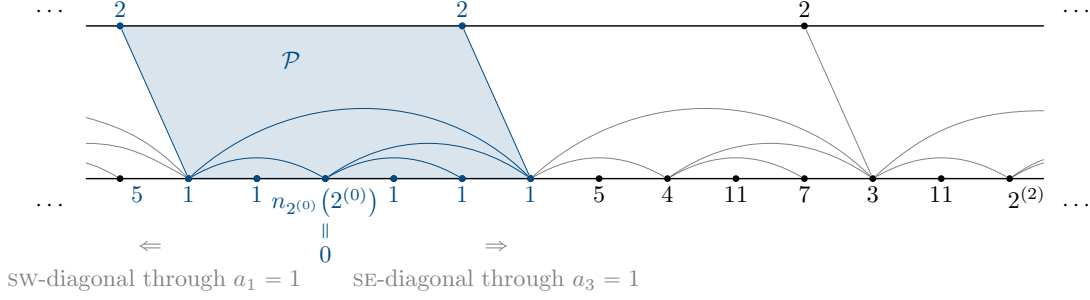

Figure~\ref{figexlabel} gives an example of this labeling for the $5$-periodic triangulation of $\mathcal{U}_5$ shown in Figure~\ref{figextriangulationstrip5p1}, with starting vertex $v=2^{(0)}$. The two sequences $(0,1,1,1,5,4,11,7,3,\cdots)$ and $(0,1,1,5,9,4,7,3,\cdots)$ are obtained for $w\geq 2^{(0)}$ and $w\leq 2^{(0)}$, respectively. Comparing these numbers with those in the associated $5$-arithmetic frieze, see Figure~\ref{figexfrieze5p2}, we observe that the first sequence is the {\sc se}-diagonal through $a_3=1$, and the second sequence gives the {\sc sw}-diagonal through $a_1=1$.
 
\begin{rem}\label{remlabelspecialvertex}
Clearly, if a special vertex $x^{(k)}$ with respect to $n$-periodic triangulation $\mathcal{T}_ n$ of $\mathcal{U}_n$ is not equal to the starting vertex $v$, we have
$$n_v\!\left(x^{(k)}\right)=n_v\!\left((x-1)^{(k)}\right)+n_v\!\left((x+1)^{(k)}\right).$$
 \end{rem}
 
The next lemma gives a useful relation if the starting vertex is special.
 
 \begin{lem}\label{lemlabelspecialvertex}
Let $\mathcal{T}_ n$ be an $n$-periodic triangulation of $\mathcal{U}_n$, $\mathcal{T}_ n\ne \mathcal{T}_ \ast$. Let $x^{(l)}$ be a special vertex with respect to $\mathcal{T}_ n$. Then for any vertex $i^{(k)}\ne x^{(l)}$ on the lower boundary of $\mathcal{U}_n$
$$n_{x^{(l)}}\!\left(i^{(k)}\right)=n_{(x-1)^{(l)}}\!\left(i^{(k)}\right)+n_{(x+1)^{(l)}}\!\left(i^{(k)}\right).$$
 \end{lem}
 
\begin{proof}
We proof the result by induction on $n$. If $n=1$, there is only the star-triangulation. For $n=2$, there is only one $2$-periodic triangulation $\mathcal{T}_2\ne \mathcal{T}_\ast$ of $\mathcal{U}_2$ up to the labeling of the vertices. W.l.o.g.\ we choose the quiddity sequence of $\mathcal{T}$ to be $q_{\mathcal{T}}=(3,1)$, see figure below.
 
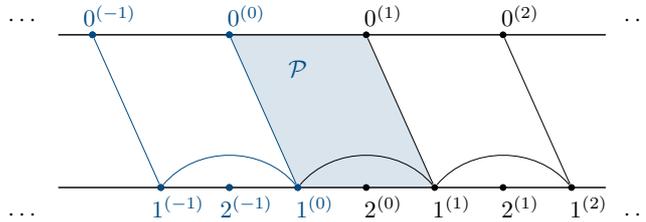
\begin{center}
\scalebox{0.9}{\begin{tikzpicture}[font=\normalsize]

\draw (6.5,1.25) node [above,shift={(0.5cm,0cm)}] {$\cdots$};
\draw (-1.5,1.25) node [above,shift={(-0.5 cm,0cm)}] {$\cdots$};
\draw[semithick] (-1.5,1.25) to (6.5,1.25);

\node (a) at (-1,1.25) [fill,myblue,circle,inner sep=1pt] {};
\node (b) at (1,1.25) [fill,myblue,circle,inner sep=1pt] {};
\node (c) at (3,1.25) [fill,circle,inner sep=1pt] {};
\node (d) at (5,1.25) [fill,circle,inner sep=1pt] {};

\draw (a) node [myblue,above,shift={(0.25cm,0cm)}] {$0^{(-1)}$};
\draw (b) node [myblue,above,shift={(0.25cm,0cm)}] {$0^{(0)}$};
\draw (c) node [above,shift={(0.25cm,0cm)}] {$0^{(1)}$};
\draw (d) node [above,shift={(0.25cm,0cm)}] {$0^{(2)}$};

\foreach \x in {0,1,2,3,4,5,6} {
   \begin{scope}[shift={(\x cm, 0 cm)}]
    \node (\x) at (0,-1) [fill,circle,inner sep=1pt] {};
   \end{scope}
}

\draw (6.5,-1) node [below,shift={(0.5cm,-0.2cm)}] {$\cdots$};
\draw (-1.5,-1) node [below,shift={(-0.5 cm,-0.2 cm)}] {$\cdots$};
\draw[semithick] (-1.5,-1) to (6.5,-1);

\draw (0) node [myblue,below,shift={(0.25cm,0cm)}] {$1^{(-1)}$};
\draw (1) node [myblue,below,shift={(0.25cm,0cm)}] {$2^{(-1)}$};
\draw (2) node [myblue,below,shift={(0.25cm,0cm)}] {$1^{(0)}$};
\draw (3) node [below,shift={(0.25cm,0cm)}] {$2^{(0)}$};
\draw (4) node [below,shift={(0.25cm,0cm)}] {$1^{(1)}$};
\draw (5) node [below,shift={(0.25cm,0cm)}] {$2^{(1)}$};
\draw (6) node [below,shift={(0.25cm,0cm)}] {$1^{(2)}$};

\draw[opacity=0, fill=myblue,fill opacity=0.15] (2,-1) -- (1,1.25) -- (3,1.25) -- (4,-1) -- cycle;
\foreach \x in {0,1,2} {  
    \node  at (\x)  [myblue,fill,circle,inner sep=1pt] {};
}
\node at (a) [fill,myblue,circle,inner sep=1pt] {};
\node at (b) [fill,myblue,circle,inner sep=1pt] {};
\draw (a) node [myblue,shift={(3cm,-0.5cm)}] {$\mathcal{P}$};

\draw[myblue,thin] (0) to (a);
\draw[myblue,thin] (2) to (b);
\draw[myblue,thin,out=50,in=130] (0) to (2);

\draw[thin] (4) to (c);
\draw[thin] (6) to (d);
\draw[thin,out=50,in=130] (2) to (4);
\draw[thin,out=50,in=130] (4) to (6);

\end{tikzpicture}}
\end{center}

\noindent In a fundamental domain there is only one special vertex, namely  $2^{(k)}$ for some $k\in\mathbb{Z}$. It is enough to show the the claim for $x^{(l)}=2^{(0)}$. One easily checks that $n_{2^{(0)}}\!\left(1^{(k)}\right)=\abs{2k-1}$ and $n_{2^{(0)}}\!\left(2^{(k)}\right)=\abs{4k}$ for all $k\in \mathbb{Z}$. Moreover, we have $n_{1^{(0)}}\!\left(1^{(k)}\right)=\abs{k}$ and $n_{1^{(0)}}\!\left(2^{(k)}\right)=\abs{2k+1}$ for all $k\in \mathbb{Z}$. Finally, $n_{1^{(1)}}\!\left(1^{(k)}\right)=\abs{k-1}$ and $n_{1^{(1)}}\!\left(2^{(k)}\right)=\abs{2k-1}$ for all $k\in \mathbb{Z}$. Hence, for $k\in\mathbb{Z}$, we have $n_{1^{(0)}}\!\left(1^{(k)}\right)+n_{1^{(1)}}\!\left(1^{(k)}\right)=\abs{2k-1}=n_{2^{(0)}}\!\left(1^{(k)}\right)$, and if $k\ne 0$, $n_{1^{(0)}}\!\left(2^{(k)}\right)+n_{1^{(1)}}\!\left(2^{(k)}\right)=\abs{4k}=n_{2^{(0)}}\!\left(2^{(k)}\right)$ as desired.

Now, we assume the claim holds for every $n$-periodic triangulation $\mathcal{T}_ n\ne \mathcal{T}_ \ast$ of $\mathcal{U}_n$. Let $\mathcal{T}_{n+1} \ne \mathcal{T}_ \ast$ be an $(n\!+\!1)$-periodic triangulation of $\mathcal{U}_{n+1}$. 

\begin{inparaenum}[\emph{\text{Case}} 1:] 
\item There are at least two special vertices in a fundamental domain of $\mathcal{T}_{n+1}$. In this case, we choose a special vertex $x^{(l)}$ for $\mathcal{T}_{n+1}$. Let $y^{(l)}\ne x^{(l)}$ be another special vertex for $\mathcal{T}_{n+1}$  and consider the $n$-periodic triangulation $\mathcal{T}_n=\mathcal{T}_{n+1\setminus y}$ of $\mathcal{U}_n$ (Corollary~\ref{corgluecutstrip}). Now, let $i^{(k)}\ne x^{(l)}$ be any vertex  on the lower boundary of $\mathcal{U}_{n+1}$. 

\begin{inparaenum}[-] 
\item If $i\ne y$, we have the corresponding vertex $i_y^{(k)}$ on the lower boundary of $\mathcal{U}_n$, where $i_y=i$ if $1\le i< y$ and $i_y=i-1$ if $y<i \le n+1$ (as in Corollary~\ref{corgluecutstrip}). Similarly, $x_y^{(l)}$ is the vertex on the lower boundary of $\mathcal{U}_n$ corresponding to $x_y^{(l)}$. Clearly, $i_y^{(k)}\ne x_y^{(l)}$. Now, since both $x$ and $i$ are different from $y$, we have $n_{x^{(l)}}\!\left({i}^{(k)}\right)=n_{x_y^{(l)}}\!\left(i_y^{(k)}\right)$. Moreover, $x$ is not a neighbor of $y$, so similarly
$$n_{{(x-1)}^{(l)}}\!\left({i}^{(k)}\right)=n_{(x_y-1)^{(l)}}\!\left(i_y^{(k)}\right), n_{{(x+1)}^{(l)}}\!\left( {i}^{(k)}\right)=n_{(x_y+1)^{(l)}}\!\left( i_y^{(k)}\right).$$
By using this and the inductive hypothesis, we get 
$$n_{x^{(l)}}\!\left({i}^{(k)}\right)=n_{(x_y-1)^{(l)}}\!\left(i_y^{(k)}\right)+n_{(x_y+1)^{(l)}}\!\left( i_y^{(k)}\right)=n_{{(x-1)}^{(l)}}\!\left({i}^{(k)}\right)+n_{{(x+1)}^{(l)}}\!\left( {i}^{(k)}\right).$$

\item Otherwise, if $i=y$, then $i^{(k)}=y^{(k)}$ is a special vertex for $\mathcal{T}_{n+1}$. Since $x\ne y$, Remark~\ref{remlabelspecialvertex} tells us that
$$n_{x^{(l)}}\!\left(y^{(k)}\right)=n_{{x}^{(l)}}\!\left((y-1)^{(k)}\right)+n_{{x}^{(l)}}\!\left( (y+1)^{(k)}\right).$$
Using the result for $i\ne y$ shown above, it follows that
\begin{align*}
n_{x^{(l)}}\!\left(y^{(k)}\right)=&\; n_{{(x-1)}^{(l)}}\!\left((y-1)^{(k)}\right)+n_{{(x+1)}^{(l)}}\!\left((y-1)^{(k)}\right)+n_{{(x-1)}^{(l)}}\!\left( (y+1)^{(k)}\right)\\
&+n_{{(x+1)}^{(l)}}\!\left( (y+1)^{(k)}\right)=n_{{(x-1)}^{(l)}}\!\left(y^{(k)}\right)+n_{{(x+1)}^{(l)}}\!\left(y^{(k)}\right),
\end{align*}
the latter follows again by Remark~\ref{remlabelspecialvertex}.
\end{inparaenum}

\item If there is only one special vertices in a fundamental domain of $\mathcal{T}_{n+1}$, they are all. In this case, the proof works similar as in the induction step and we leave it to the reader to check the details.
\end{inparaenum}
 \end{proof}
 
The next result shows that this labeling algorithm provides all entries occurring in an  \mbox{$n$-arithmetic} frieze $\mathcal{F}_n$ associated to a given \mbox{$n$-periodic} triangulation of $\mathcal{U}_n$, and hence to every triangulation of $S_n^1$.

\begin{thm}\label{thmaltderscription}
Let $\mathcal{T}_n$ be an $n$-periodic triangulation of $\mathcal{U}_n$ with associated $n$-arithmetic frieze\linebreak $\mathcal{F}_{\mathcal{T}_n}=(m_{ij})_{j-i\ge -2}$. Then 
$$m_{ij}=n_{(i-1)^{(k)}}\!\left((j+1)^{(k)}\right).$$
\end{thm}

\begin{proof}
It is enough to show the claim for a fundamental domain. W.l.o.g.\ we choose $k$ to be $0$. First we show the claim for the star-triangulation $\mathcal{T}_\ast$ of $\mathcal{U}_n$ with basic infinite frieze $\mathcal{F}_\ast=(m_{ij})_{j-i\ge -2}$, where $m_{ij}=j-i+2$. Let $i^{(0)}$ be a fixed vertex on the lower boundary of $\mathcal{U}_n$, $i\in\{1,2,\dots, n\}$. Then we have $n_{i^{(0)}}\!\left(i^{(0)}\right)=0=m_{i+1,i-1}$, $n_{i^{(0)}}\!\left((i-1)^{(0)}\right)=n_{i^{(0)}}\!\left((i+1)^{(0)}\right)=1=m_{i+1,i}$ and $n_{i^{(0)}}\!\left(0^{(k)}\right)=1$ for all $k\in\mathbb{Z}$. Recall, for $j=k+ln$ with $k\in\{1,2,\dots,n\}$ and $l\in\mathbb{Z}$, we actually have $j^{(0)}=k^{(l)}$. Clearly, for $j\in \mathbb{N}$ such that $j \ge i+1$, we have $n_{i^{(0)}}\!\left(j^{(0)}\right)=n_{i^{(0)}}\!\left((j-1)^{(0)}\right)+n_{i^{(0)}}\!\left(0^{(0)}\right)$, and inductively $n_{i^{(0)}}\!\left(j^{(0)}\right)=j-i=m_{i+1,j-1}$ for $i+1\in\{2,3,\dots,n+1\}$ and $j-1\ge i+1$. Hence the claim holds for star-triangulations, in particular, for triangulations of $\mathcal{U}_1$.

For the remaining triangulations of $\mathcal{U}_n$, we proceed by induction on $n$ and assume the claim is true for $n\geq 1$. Let $\mathcal{T}_{n+1}\ne \mathcal{T}_\ast$ be an \mbox{$(n\!+\!1)$-periodic} triangulation of $\mathcal{U}_{n+1}$ with quiddity sequence $q_{\mathcal{T}_{n+1}}=(a_1,a_2,\dots,a_{n+1})$ and associated $(n\!+\!1)$-arithmetic frieze $\mathcal{F}_{\mathcal{T}_{n+1}}=(m_{ij})_{j-i\ge -2}$. We know that $\mathcal{T}_{n+1}$ contains a special vertex $x^{(0)}$ (Corollary~\ref{corspecialvertex}) and $\mathcal{T}_n:=\mathcal{T}_{n+1\setminus{x}}$ is an $n$-periodic triangulation of $\mathcal{U}_{n}$ with quiddity sequence $q_{\mathcal{T}_n}=(\check a_1,\check a_2,\dots,\check a_n)$ as given in Corollary~\ref{corgluecutstrip}. Let $\mathcal{F}_{\mathcal{T}_n}=(\check m_{ij})_{i\geq 0,j \in \mathbb{Z}}$ denote the $n$-arithmetic frieze associated to $\mathcal{T}_n$ and recall that $\mathcal{F}_{\mathcal{T}_{n+1}}$ is obtained from $\mathcal{F}_{\mathcal{T}_n}$ by $n$-gluing above $(\check a_{x-1},\check a_x)$ in $\mathcal{F}_{\mathcal{T}_n}$ (Corollary~\ref{corspecialvertexfriezeequivalce}).

Using the notation of Corollary~\ref{cornglueperiodicfrieze}, for $i\ne x$, $i_x^{(k)}$ denotes the vertex of $\mathcal{U}_n$ corresponding to the vertex $i^{(k)}$ of $\mathcal{U}_{n+1}$. To avoid confusion, we denote the attached labels for $\mathcal{T}_{n+1}$ and $\mathcal{T}_n$ respectively by $n_{\ast}(-)$ and $\check n_{\ast}(-)$. For two vertices $i^{(k)},j^{(l)}$ on the lower boundary of $\mathcal{U}_{n+1}$, it is easy to see that
$$n_{i^{(k)}}\!\left((j^{(l)}\right)=\check n_{i_x^{(k)}}\!\left(j_x^{(l)}\right)\quad \text{if } i,j\ne x.$$

By definition, for $i\in\{1,2,\dots,n+1\}$, we have $m_{i,i-2}=0=n_{(i-1)^{(0)}}\!\left((i-1)^{(0)}\right)$ and $ m_{i,i-1}=1=n_{(i-1)^{(0)}}\!\left(i^{(0)}\right)$. Now, we consider $m_{ij}$ in $\mathcal{F}_{\mathcal{T}_{n+1}}$ with $i\in\{1,2,\dots,n+1\}$ and $j\ge i$. We apply Corollary~\ref{cornglueperiodicfrieze} (for $k=x-1$) and distinguish the following four cases as in the proof of Theorem~\ref{thmmatchings}.

\begin{inparaenum}[\emph{\text{Case }}1:] 
\item Suppose $i\not\equiv x+1$ and $j\not \equiv  x-1$. Using the indcutvie hypothesis, we get $m_{ij}=\check m_{i_{x+1},j_{x-1}}=\check n_{(i_{x+1}-1)^{(0)}}\!\left((j_{x-1}+1)^{(0)}\right)=\check n_{(i-1)_x^{(0)}}\!\left((j+1)_x^{(0)}\right)=n_{(i-1)^{(0)}}\!\left((j+1)^{(0)}\right) $ as desired.

\item Let $i\not \equiv  x+1$ and $j\equiv x-1$. In this case, by induction, we get  $m_{i,x-1}=\check m_{i_{x+1},x-2}+\check m_{i_{x+1},x-1}=\check n_{(i-1)_x^{(0)}}\!\left((x-1)^{(0)}\right)+\check n_{(i-1)_x^{(0)}}\!\left(x^{(0)}\right)=n_{(i-1)^{(0)}}\!\left((x-1)^{(0)}\right)+n_{(i-1)^{(0)}}\!\left((x+1)^{(0)}\right)=n_{(i-1)^{(0)}}\!\left(x^{(0)}\right)$, where the latter follows by Ramark~\ref{remlabelspecialvertex} since $(i-1)^{(0)}\ne x^{(l)}$, as otherwise we would get $j=i-1$, a contradiction to $j\ge i$.

\item Let $i\equiv x+1$ and $j\not \equiv x-1$. Closely analogous to Case $2$, using Lemma~\ref{lemlabelspecialvertex} instead of Ramark~\ref{remlabelspecialvertex}, we have $m_{x+1,j}=\check m_{x,j_{x-1}}+\check m_{x+1,j_{x-1}}=\check n_{(x-1)^{(0)}}\!\left((j+1)_x^{(0)}\right)+\check n_{x^{(0)}}\!\left((j+1)_x^{(0)}\right)= n_{(x-1)^{(0)}}\!\left((j+1)^{(0)}\right)+ n_{(x+1)^{(0)}}\!\left((j+1)^{(0)}\right)= n_{x^{(0)}}\!\left((j+1)^{(0)}\right)$, where $x^{(0)}\ne (j+1)^{(0)}$ since \mbox{$j\ne x-1$}.

\item Finally, if $i\equiv x+1$ and $j\equiv x-1$, we apply Corollary~\ref{cordiagonalsum}, Case~$3$ and Remark~\ref{remlabelspecialvertex} and it follows that
$m_{x+1,x-1}=m_{x+1,x-2}+m_{x+1,x}=n_{x^{(0)}}\!\left((x-1)^{(0)}\right)+n_{x^{(0)}}\!\left((x+1)^{(0)}\right)=n_{x^{(0)}}\!\left(x^{(l)}\right)$.
\end{inparaenum}
This completes the proof.
\end{proof}

Clearly, we can apply the labeling algorithm to an $n$-periodic triangulation $\mathcal{T}_n$ of $\mathcal{U}_n$ for a starting vertex $v=0^{(k))}$ with $k\in\mathbb{Z}$ on the upper boundary of $\mathcal{U}_n$. Then the labels on the lower boundary provide an $n$-periodic sequence determined by $(n_1,n_2,\dots, n_n)$, where $n_i=n_v\!\left(i^{(k)}\right)$. Note that this sequence does not depend on the choice of $k$.

\begin{prop}\label{propcommondifferences}
Let $\mathcal{T}$ be an $n$-periodic triangulation of $\mathcal{U}_n$ with $1\leq r\leq n$ bridging arcs, up to translates. Let $\mathcal{F}_n=(m_{ij})_{j-1\ge -2}$ be the $n$-arithmetic frieze associated to $\mathcal{T}$ with common differences $d_{ik}=m_{i,(i+k-3)+n}-m_{i,i+k-3}$. Then $d_{ik}=r n_{i-1} n_{i+k-2}$, where $i-1$ and $i+k-2$ are reduced modulo $n$. 
\end{prop}

\begin{proof}
Let $\mathcal{F}_n$ be the basic frieze with star-triangulation $\mathcal{T}_\ast$ of $\mathcal{U}_n$. Then we have $r=n$ and $n_i=1$ for all $i$, so $d_{ik}=n$ as desired. In particular, the claim holds if $n=1$. We proceed by induction on $n$ and assume that the claim is true for $n\ge1$. We consider an $(n\!+\!1)$-periodic triangulation $\mathcal{T}_{n+1}$ ($\ne \mathcal{T}_\ast$) of $\mathcal{U}_{n+1}$ with $1\leq r\leq n$ bridging arcs, up to translates, and quiddity sequence $q_{\mathcal{T}_{n+1\setminus x}}=(a_1,a_2,\dots.a_{n+1}).$ Let $\mathcal{F}_{n+1}=(m_{ij})_{j-i\ge -2}$ be the $(n\!+\!1)$-arithmetic frieze associated to $\mathcal{T}_{n+1}$. By assumption, there exists a special vertex $x^{(0)}$ in $\mathcal{T}_{n+1}$ (Corollary~\ref{corspecialvertexfriezeequivalce}) and $\mathcal{T}_n=\mathcal{T}_{n+1\setminus x}$ is an $n$-periodic triangulation of $\mathcal{U}_n$ with quiddity sequence $q_{\mathcal{T}_n}=(\check a_1,\check a_2,\dots,\check a_n)$ (Corollary~\ref{corgluecutstrip}) and associated $n$-arithmetic frieze $\mathcal{F}_n=(\check m_{ij})_{i\geq 0,j \in \mathbb{Z}}$. Clearly, the number of bridging arcs, up to translates, stays unchanged, so $\mathcal{T}_n$ has also $r$ bridging arcs. Let $n_\ast$ and $\check n_\ast$ denote the labels attached to $\mathcal{T}_{n+1}$ and $\mathcal{T}_n$, respectively. Then $n_x=n_{x-1}+n_{x+1}$ and $n_i=\check n_{i_x}$ for $i\in\{1,2,\dots n+1\}\setminus\{x\}$ (notation of Corollary~\ref{cornglueperiodicfrieze}). Corollary~\ref{corspecialvertexfriezeequivalce} implies that $n$-gluing above $(\check a_{x-1},\check a_x)$ in $\mathcal{F}_n$ leads back to $\mathcal{F}_{n+1}$ and we can use Corollary~\ref{cornglueperiodicfrieze} (for $k=x-1$) to express the entries in $\mathcal{F}_{n+1}$ in terms of entries in $\mathcal{F}_n$. In particular, we can express the common differences $d_{ik}$ for $\mathcal{F}_{n+1}$ by the common differences $\check d_{ik}$ for $\mathcal{F}_n$ (as in the proof of Proposition \ref{proptriangulationfriezearithmetic}). As usual, we must  check the following four separate cases. 

\begin{inparaenum}[\emph{\text{Case }}1:] 
\item If $i\not\equiv x+1$ and $k\not\equiv x-i+2$ (modulo $n+1$), we have $d_{ik}=\check d_{i_{x+1}\check k}$, where $\check k=(i+k-3)_{x-1}-i_{x+1}+3$. By our inductive hypothesis, we deduce $d_{ik}=r\check n_{i_{x+1}-1} \check n_{(i+k-3)_{x-1}+1}=r\check n_{(i-1)_x} \check n_{(i+k-2)_x}=r n_{i-1} n_{i+k-2}$.

\item Let $i\not \equiv  x+1$ and $k\equiv x-i+2$. In this case, we have $d_{ik}=\check d_{i_{x+1}\check k_1+}+\check d_{i_{x+1}\check k_2}$ with $\check k_1=(i+k-3)_{x-1}-i_{x+1}+2$ and $\check k_2=(i+k-3)_{x-1}-i_{x+1}+3$. By induction, it follows that $d_{ik}=r\check n_{i_{x+1}-1} \check n_{(i+k-3)_{x-1}}+r\check n_{i_{x+1}-1} \check n_{(i+k-3)_{x-1}+1}=r\check n_{(i-1)_x}\left(\check n_{(x-1)_{x-1}}+\check n_{(x-1)_{x-1}+1} \right)=r n_{i-1}\left(n_{x-1}+ n_{x+1} \right)=r n_{i-1}n_x$, where $x\equiv i+k-2$.

\item Suppose $i\equiv  x+1$ and $k\not\equiv x-i+2$. Then $d_{ik} =\check d_{i_{x+1}-1,\check k}+\check d_{i_{x+1}\check k}$ with $\check k=(i+k-3)_{x-1}-i_{x+1}+3$, so by the inductive hypothesis $d_{ik}=r(\check n_{i_{x+1}-2}+\check n_{i_{x+1}-1}) \check n_{(i+k-3)_{x-1}+1}=r(\check n_{x-1}+\check n_x) \check n_{(i+k-2)_x}=r(n_{x-1}+n_{x+1}) n_{i+k-2}=r n_{i-1} n_{i+k-2}$.

\item Finally, let $i\equiv  x+1$ and $k\equiv 1$. In this case we use Corollary~\ref{cordiagonalsum} and deduce that $d_{ik}=d_{i,k-1}+d_{i,k+1}$. By Case $3$, we obtain $d_{ik}=r n_{i-1}\left(n_{x-1}+n_{x+1}\right)=r n_{i-1}n_{i+k-2}$ as desired.
\end{inparaenum}
\end{proof}

\subsection*{Acknowledgements}
A special gratitude goes to Karin Baur who gave many inputs and shared various valuable comments for the present article. It is also a pleasure to thank the members of the research group for Algebra and Number Theory at the Institute for Mathematics and Scientific Computing of the University of Graz for several hours of discussions. Finally, I gratefully acknowledge support from NAWI Graz.

\textsc{ University of Graz, Institute for mathematics and scientific computing, Heinrichstrasse 36, 8010 Graz, Austria.}

\emph{E-mail address:}
\email{manuela.tschabold@uni-graz.at}

\end{document}